\documentclass[12pt]{amsart}

\usepackage{amssymb}
\usepackage[T1]{fontenc}
\usepackage{lmodern}

\newtheorem{theorem}{Theorem}[section]

\newtheorem{lemma}[theorem]{Lemma}

\newtheorem{conjecture}{Conjecture}

\theoremstyle{definition}

\newtheorem{remark}[theorem]{Remark}

\numberwithin{equation}{section}

\newcommand{\PSL}{\mathrm{PSL}}

\newcommand{\ZZ}{\mathbb{Z}}

\newcommand{\RR}{\mathbb{R}}

\newcommand{\HH}{\mathbb{H}}

\newcommand{\PSLR}{\PSL_2(\RR)}
\newcommand{\PSLZ}{\PSL_2(\ZZ)}

\title{Local square mean in the hyperbolic circle problem and sums of Salié sums}

\author{
    András Biró\\
    HUN-REN A. Rényi Institute of Mathematics\\
    \texttt{biro.andras@renyi.hu}
}

\thanks{Research partially supported by the NKFIH (National Research, Development and Innovation Office) Grants No. K135885, K143876, and by the Rényi Intézet Lendület Automorphic Research Group}


\begin{document}

\begin{abstract}
Let $\Gamma\subseteq \PSLR$ be a finite volume Fuchsian group. The hyperbolic circle problem is the
estimation of the number of elements of the $\Gamma$-orbit of
$z$ in a hyperbolic circle around $w$ of radius $R$, where $z$ and
$w$ are given points of the upper half plane and $R$ is a
large number. An estimate with error term $e^{\frac 23R}$ is
known. Recently, taking $
z=w$ and considering $\Gamma =\PSLZ$, we have shown the estimate $
e^{\left(\frac 9{14}+\epsilon\right)R}$  for the local $L^2$-norm of the error term, which is better than the pointwise bound. Here we improve the exponent
$\frac 9{14}$, conditionally on a twisted Linnik-Selberg-type conjecture for sums of Salié sums.

\end{abstract}

\maketitle

\markboth{András Biró}{Local square mean in the hyperbolic circle problem and sums of Salié sums}

\section{Introduction}

\subsection{Statement of the main result.} Let $\HH$ be the upper half plane. For $
z,w\in \HH$ let
\begin{equation}u(z,w)=\frac {\left|z-w\right|^2}{4\mbox{\rm Im$z$Im$
w$}},\label{1.1}\end{equation}
this is closely related to the hyperbolic distance $\rho (z,w)$ of $
z$
and $w$, namely we have $1+2u=\cosh\rho$. The elements of the
group $\PSLR$ act on $\HH$ by linear fractional transformations, these are isometries of the
hyperbolic plane. Let $d\mu_z=\frac {dxdy}{y^2}$, this measure is invariant with respect to
the action of $\PSLR$ on $\HH$. Let $\Gamma =\PSLZ$. For $z,w\in\HH$ and $X>2$ define
$$N\left(z,w,X\right):=\left|\left\{\gamma\in\Gamma :\hbox{\rm \ $4
u\left(\gamma z,w\right)+2\le X$}\right\}\right|,$$
this is the number of points $\gamma z$ in the hyperbolic circle
around $w$ of radius $\cosh^{-1}\left(X/2\right)$, so the estimation of this
quantity is called the hyperbolic circle problem. We know that
\[\left|N\left(z,w,X\right)-3X\right|=O_{z,w}\left(X^{\frac 23}\right),\]
this is an unpublished theorem of Selberg. This estimate has not been improved for any guven pair of points until the recent paper \cite{CCLR}, where a logarithmic improvement was proved for Heegner points of different duscriminants.

From now on we take $z=w$, and we write $N\left(z,X\right):=N\left(z,z,X\right)$. Let ${\mathcal F}$ be
the closure of the standard fundamental domain of $\Gamma$, i.e.
\begin{equation}\mathcal F=\left\{z\in {\bf C}:\;\mbox{\rm Im$z>0$}
,\>-\frac 12\le\mbox{\rm Re$z$}\le\frac 12,\>\left|z\right|\ge 1\right
\}.\label{1.3}\end{equation}
In \cite[Theorem 1.1]{Biro} we proved the following
theorem.
\begin{theorem}\label{thm:1}
Let $\Gamma =\PSLZ$, let $\mathcal F$ be as in \eqref{1.3}, and let
$\Omega\subseteq\mathcal F$ be a compact set, then for any $c>\frac
9{14}$ we have
\begin{equation}\left(\int_{\Omega}\left(N\left(z,X\right)-3X\right
)^2d\mu_z\right)^{1/2}=O_{\Omega ,c}\left(X^c\right).\label{1.4}\end{equation}
\end{theorem}

The goal of the present paper is to show that it is
possible to improve the exponent in this theorem assuming
the following conjecture.

If $x,c$ are integers with $\gcd\left(x,c\right)=1$, then $\overline
x$ denotes the
multiplicative inverse of $x$ modulo $c$. For odd $c>0$ and any integers $
m,n$ introduce the Salié sums
\[T\left(m,n;c\right):=\sum_{x\:{\rm m}{\rm o}{\rm d}\:c,\gcd\left
(x,c\right)=1}\left(\frac xc\right)e\left(\frac {m\overline x+nx}
c\right),\]
where $\left(\frac xc\right)$ is the Jacobi symbol.

\begin{conjecture}\label{conj:1} {\normalfont (Twisted Linnik-Selberg for
Salié sums).} For every $\epsilon >0$ there is a $\delta >0$ such
that the following statement is true. If $C\ge 1$, $f$ is a
smooth function on the positive real axis vanishing
outside the interval $\left[C,2C\right]$, $L,K$ and $r$ are positive integers
satisfying $2|L$ and $\gcd\left(L,Kr\right)=1$, $B>0$ is a real number and we have
\[LK\left(1+B\right)\ll_{\delta}C^{\delta},\;x^jf^{\left(j\right)}\left
(x\right)\ll_{\delta ,j}C^{j\delta}\]
for every $x>0$ and every integer $j\ge 0$, then for every positive integers $
m,n$ such that $mnL$ is not a square and for every $\alpha\in\left[-B,B\right]$ we have
\[\frac 1C\sum_{c\equiv r\left({\rm m}{\rm o}{\rm d}\,L\right),\,
K|c}f\left(c\right)T\left(m,\overline Ln;c\right)e\left(\frac {\sqrt {
mn}}c\alpha\right)\ll_{\epsilon ,\delta}\left(mnC\right)^{\epsilon}
.\]
\end{conjecture}

This is analogous to \cite[Conjecture 1.1]{Br}, which
concerns Kloosterman sums instead of Salié sums.
The main difference is that we include here also a
multiplicative inverse $\overline L$ of an integer $L\ll_{\delta}
C^{\delta}$. It is
also a difference that $B$ and the modulus $k:=LK$ are
bounded in \cite[Conjecture 1.1]{Br}, but here we allow
them to grow very slowly with $C$, i.e. we assume that
these parameters are $\ll_{\delta}C^{\delta}$. Note also that we have a
smooth summation here instead of a sharp cut-off.

The main result of this paper is the following theorem.

\begin{theorem}\label{thm:2}Assume that Conjecture
\ref{conj:1} is true. Then, using the notation and
assumption of Theorem \ref{thm:1} we have \eqref{1.4}
with some $c<9/14$.
\end{theorem}

\begin{remark}\label{rem:2}
In fact a weaker assumption than Conjecture
\ref{conj:1} would be enough for the same conclusion. We
do not give the details here in the Introduction, but see the Statement in Lemma \ref{lem:13}.

Note that in the first version of this paper (arxiv:2604.11205) we stated Conjevture \ref{conj:1} without the assumption that $mnL$ is not a square. It was pointed out to me by Floris Alderliefste that the conjecture as was stated in the first version is false, it contradicts \cite[Theorem 1]{L}. For the details see Statement A.5.7 on p 37 of \cite{A}. Conjecture
\ref{conj:1} as stated in the present paper is still enough for the proof of the Statement in Lemma \ref{lem:13}. This fact is proved in Lemma \ref{lem:14}.
\end{remark}

\begin{remark}It would be possible to compute an explicit value $c$ in Theorem \ref{thm:2} following the steps of the proof. However, we could not get the optimal value, and it would be very tedious, so we do not specify a value of $c$.
\end{remark}

\begin{remark}The statement of Theorem \ref{thm:1} is
valid for any finite index subgrup $\Gamma$ of $\PSLZ$ (including
noncongruence subgroups) in place of $\PSLZ$. Indeed, the proof of this theorem given in
\cite{Biro} uses only an upper bound for the class
numbers of pairs of quadratic forms, hence that proof
generalizes easily for every finite index subgroup of
$\PSLZ$. The main novelty of this paper is that instead of
an upper bound we apply the explicit formula proved for these class numbers in
\cite{biroClass} (this possibility was mentioned already at
the end of \cite[Subsection 1.2]{Biro}). Such an explicit formula is not known in
the case of finite index subgruoups, so we cannot
improve the exponent $9/14$ even conditionally in the case
of finite index subgrups of $\PSLZ$.
\end{remark}

\begin{remark}This is not the first instance when a Linnik-Selberg type conjecture occurs in connection with the hyperbolic circle problem. Indeed, the
local average of the hyperbolic circle problem was studied in \cite{PR}, and it
was shown in Remark 6.4 there that a conjecture on a spectral exponential sum implies
the optimal bound in that problem. Such a conjecture was shown in Proposition 1.4 of
\cite{BFR} to follow
from a Linnik-Selberg type conjecture on additively twisted sums of Kloosterman
sums, originally proposed by Iwaniec.

\end{remark}

\begin{remark}If we write a general finite volume Fuchsian group
in place of $\PSLZ$, we cannot prove a better bound for
the left-hand side of \eqref{1.4} than $X^{2/3}$, which follows from the pointwise bound.
\end{remark}

\begin{remark}Partial result towards \cite[Conjecture 1.1]{Br} are
proved in \cite{St}. Probably similar partial results could
be proved in connection with our Conjecture \ref{conj:1}
using the half-integral weight Kuznetsov formula instead
of the integral weiight Kuznetsov formula. In addition,
an averaged form of Conjecture \ref{conj:1} would
suffice (see Remark \ref{rem:2}), therefore it might be
possible to prove Theorem \ref{thm:2} unconditionally by
developing an averaged version of the analogues of the
arguments in \cite{St}. It would be interesting to
investigate this possibility in the future.
\end{remark}

\subsection{Outline of the proof.} As in \cite{Biro}, we take an integer $
J\ge 2$, it will
be fixed but large. We also take a parameter $d$, we assume $X^{\frac
23+\frac 1{100}}\le d\le X^{99/100}$, and we consider the sum
\[N_{d,J}\left(z,X\right):=\sum_{j=0}^J\left(-1\right)^j\left(\begin{array}{c}
J\\
j\end{array}
\right)\int_1^2\eta_0\left(\tau\right)N\left(z,X-jd\tau\right)d\tau
,\]
where $\eta_0$ is a given nonnegative smooth function on $\left(0
,\infty\right)$ such that $\eta_0\left(\tau\right)=0$
for $\tau\notin\left[1,2\right]$, and $\int_1^2\eta_0\left(\tau\right
)d\tau =1.$ It is proved in \cite{Biro}
that we have with any $\epsilon >0$ that
\[N\left(z,X\right)-3X=N_{d,J,{\rm h}{\rm y}{\rm p}}\left(z,X\right
)+O_{\Omega ,\epsilon}\left(\frac X{\sqrt {d}}+X^{\frac 12+\epsilon}\right
)\]
for $z\in\Omega$, where $N_{d,J,{\rm h}{\rm y}{\rm p}}\left(z,X\right
)$ is the contribution of the hyperbolic
$\gamma\in\Gamma$ to $N_{d,J}\left(z,X\right)$. It is proved in \cite{Biro} that
\begin{equation}\int_{\mathcal F}\left(N_{d,J,{\rm h}{\rm y}{\rm p}}\left
(z,X\right)\right)^2d\mu_z\ll_{\Omega ,\epsilon}X^{\epsilon}\frac {
d^{5/2}}{\sqrt {X}},\label{67}\end{equation}
and then we get Theorem \ref{thm:1} taking $d=X^{5/7}$.

In this paper we will improve \eqref{67} conditionally.
We gave an expression in \cite{Biro} for the left-hand
side of \eqref{67} whose most essential part is an expression of type

\begin{equation}\sum_{t_1,t_2,f}h\left(t_1^2-4,t_2^2-4,f\right)\sum_{
j_1,j_2=0}^J\left(-1\right)^{j_1+j_2}\left(\begin{array}{c}
J\\
j_1\end{array}
\right)\left(\begin{array}{c}
J\\
j_2\end{array}
\right)F_{X,d,j_1,j_2}\left(t_1,t_2,f\right),\label{68}\end{equation}
where $t_1,t_2>2$ and $f$ run over integers, $F_{X,d,j_1,j_2}\left
(t_1,t_2,f\right)$ is
an analytic expression, and $h\left(t_1^2-4,t_2^2-4,f\right)$ has the
following arithmetic meaning. If $d_1,d_2,t\in\mathbb Z$, then
$h\left(d_1,d_2,t\right)$ denotes the number of  $SL_2(\mathbb Z)$-equivalence
classes of pairs $\left(Q_1,Q_2\right)$ of quadratic forms $Q_i\left
(X,Y\right)=A_iX^2+B_iXY+C_iY^2$ with
integer coefficients satisfying that the discriminant of
$Q_i$ is $d_i$, and the codiscriminant $B_1B_2-2A_1C_2-2A_2C_1$ of
$Q_1$ and $Q_2$ is $t$.

It is proved in \cite{Biro} that
\begin{equation}\sum_{t_1,t_2,f}h\left(t_1^2-4,t_2^2-4,f\right)\left
|\sum_{j_1,j_2=0}^J\left(-1\right)^{j_1+j_2}\left(\begin{array}{c}
J\\
j_1\end{array}
\right)\left(\begin{array}{c}
J\\
j_2\end{array}
\right)F_{X,d,j_1,j_2}\left(t_1,t_2,f\right)\right|,\label{69}\end{equation}
which is clearly an upper bound for \eqref{68}, is
$\ll X^{\epsilon}\frac {d^{5/2}}{\sqrt {X}}$. To prove this estimate we simply use an
upper bound for $h\left(t_1^2-4,t_2^2-4,f\right)$ and not an explicit
formula.

In the present paper we identify the critical parts of
the summation over $t_1,t_2,f$ in \eqref{68}, i.e. those parts
where estimating in absolute value as in \eqref{69} we
do not get a better upper bound than $\frac {d^{5/2}}{\sqrt {X}}$. In these
critical parts we use the explicit formula proved for
$h\left(t_1^2-4,t_2^2-4,f\right)$ in \cite{biroClass}, and we prove
assuming Conjecture \ref{conj:1} that there is a cancellation in the summation over $
t_1,t_2,f$.

Very roughly speaking, the critical part consists of
sums of the following type: $a\le t_1,t_2\le 2a$ with an $a$ satisfying
\[a,\sqrt {X}-2a\gg X^{\frac 12-\epsilon},\]
and $f$ runs over an interval of length around $d$ whose
endpoints have order of magnitude $X$. We will see that
$\left(d/X\right)^{3/2}$ is an upper bound for $F_{X,d,j_1,j_2}\left
(t_1,t_2,f\right)$. The formula for $h\left(t_1^2-4,t_2^2-4,f\right
)$ will be roughly
\[\sum_{D|f^2-\left(t_1^2-4\right)\left(t_2^2-4\right)}\left(\frac {
t_1^2-4}D\right)\]
(this is not exactly true, but we will have a similar
expression). Estimating everything trivially we get the bound $\frac {d^{5/2}}{\sqrt {X}}$. Instead of this, taking complementary divisors we can
assume $D\ll X$, and for every fixed $D$ we apply
Poisson summation in $f$. The zeroth term of the Poisson
summation will be small using the summation over $j_1,j_2$,
the nonzero terms will be small for any $j_1,j_2$ assuming
Conjecture \ref{conj:1}.

\section{Reduction to the estimation of a square integral}

We recall here an estimation proved in \cite[Section
5]{Biro}.

If $t>2$ is an integer and $m$ is a compactly supported bounded function on
$[0,\infty )$, for $z\in\HH$ let
\[M_{t,m}(z)=\sum_{\gamma\in\Gamma_t}m\left(z,\gamma z\right),\]
where
\[\Gamma_t:=\left\{\begin{pmatrix}a&b\\ c&d\end{pmatrix}\in SL_2(\mathbb{
Z}),a+d=t\right\}\]
and for $z,w\in\HH$ we write
\[m(z,w)=m\left(u(z,w)\right)\]
by an abuse of notation, see \eqref{1.1} for $u(z,w)$.

For $x>0$ introduce the notation
\[k_x\left(y\right)=1\mbox{\rm \ for $0\le y\le x$},\quad k_x\left
(y\right)=0\mbox{\rm \ for $y>x$}.\]
It is proved in \cite[(5.16), (5.17)]{Biro} that if $\epsilon >0$,
$d\ge 100$ and $100Jd\le X$ with a positive integer $J\ge 2$, then
\begin{equation}\int_{\Omega}\left(N\left(z,X\right)-3X\right)^2d
\mu_z\ll_{\Omega ,\epsilon ,J}\frac {X^2}d+X^{1+\epsilon}+\int_1^
2I_{d,X,J}\left(\tau\right)d\tau ,\label{2.0}\end{equation}
where $ $
\[I_{d,X,J}\left(\tau\right):=\int_{\mathcal F}\left(\sum_{t>2}\sum_{
j=0}^J\left(-1\right)^j\left(\begin{array}{c}
J\\
j\end{array}
\right)M_{t,k_{\left(X-jd\tau -2\right)/4}}\left(z\right)\right)^
2d\mu_z.\]
It is proved in \cite[(5.18) and Section 6]{Biro} that if $\epsilon
>0$ is given and the
integer $J\ge 2$ is fixed to be large enough in terms of $\epsilon$,
then we have
\begin{equation}_{}I_{d,X,J}\left(\tau\right)\ll_{\epsilon}X^{\epsilon}\frac {
d^{5/2}}{\sqrt {X}}\label{2.1}\end{equation}
uniformly for $1\le\tau\le 2$ and $X^{2/3}\le d\le X^{99/100}$. Then we
choose $d=X^{5/7}$ there to get Theorem \ref{thm:1}.

In this paper we will improve the estimate \eqref{2.1}
assuming Conjecture \ref{conj:1}. We will show assuming
that conjecture that there is a $\beta >0$ such that
\begin{equation}I_{d,X,J}\left(\tau\right)\ll\frac {d^{5/2}}{\sqrt {
X}}X^{-\beta}\label{2.2}\end{equation}
holds uniformly for $1\le\tau\le 2$ and
\begin{equation}X^{\frac 23+\frac 1{100}}\le d\le X^{99/100}\label{201}\end{equation}
if $J$ is large enough. Then \eqref{2.2} and \eqref{2.0} imply
Theorem \ref{thm:2} by choosing $d$ to be a bit larger
power of $X$ than $ $$X^{5/7}$.

So it is enough to show the estimate \eqref{2.2}.

\section{Proof of the improved estimate assuming Conjecture \ref{conj:1}}

\subsection{Improvement of a lemma of \cite{Biro}.}

The following lemma is an improvement of \cite[(3.20)]{Biro}. We first recall a notation from \cite{Biro}. For any finite set of integers $n_1, n_2\ldots ,n_r$ we write
\[S\left(n_1,n_2,\ldots ,n_r\right)=\max\left\{k\ge 1:\;k^2|{\rm g}
{\rm c}{\rm d}\left(n_1,n_2,\ldots ,n_r\right)\right\}.\]
\begin{lemma}\label{lem:1}Let $3\le a<b\le c\le 2a$ be integers.
For any $\epsilon >0$ we have that
\begin{equation}\sum_{t_1=a}^{c-1}\sum_{a\le t_2\le c-1,0<\left|t_
2-t_1\right|\le b-a}S\left(t_1^2-4,t_2^2-4\right)\label{1}\end{equation}
is
\[\ll_{\epsilon}a^{\epsilon}\left(\left(c-a\right)\left(b-a\right
)+\sqrt a\left(c-a\right)\right).\]
\end{lemma}

\begin{proof}As in the proof of \cite[(Lemma 3.5)]{Biro}, we may assume $b=c$, since the general case
follows from this special case by dividing the summation
over $t_1$ into $O\left(\frac {c-a}{b-a}\right)$ subsums.

Let us write $k:=S\left(t_1^2-4,t_2^2-4\right)$ and $E:=k^2$. Then $
E$
divides both $t_1^2-4$ and $t_2^2-4$, hence $E|t_1^2-t_2^2$. Let
$E_1:=\gcd\left(E,t_1-t_2\right)$ and $E_2:=E/E_1$. Then $E_2|t_1
+t_2$. If $p\neq 2$ is a
common prime divisor of $E_1$ and $E_2$, then $p$ divides $t_1$
and also $k$. Hence $p|t_1^2-4$, but this would give $p|4$, a
contradiction. So only $2$ might be a common prime divisor
of $E_1$ and $E_2$. Since $E_1E_2$ is a square, this implies that
$E_1=uk_1^2$, where $u=1$ or $u=2$, and $k_1>0$ is an intger.

Note also that since $E_1|t_1-t_2$ and $E_2|t_1+t_2$, for given $
t_1$,
$E_1$ and $E_2$ the number of possible integers $t_2$ is
\[\le 1+\frac {b-a}{\max\left(E_1,E_2\right)}\le 1+\frac {b-a}{\sqrt {
E_1}\sqrt {E_2}}.\]
Then \eqref{1} is $\ll$ than
\begin{equation}\sum_{t_1=a}^{b-1}\sum_{k_1^2E_2|t_1^2-4,k_1^2\ll
b-a,E_2\ll a}k_1\sqrt {E_2}\left(1+\frac {b-a}{k_1\sqrt {E_2}}\right
),\label{2}\end{equation}
where the inner summation is over integers $k_1$ and $E_2$.

The contribution to \eqref{2} of the second term in the bracket is
clearly $\ll_{\epsilon}a^{\epsilon}\left(b-a\right)^2$ for any $\epsilon
>0$. The contribution of the
first term in the bracket is
\[\ll_{\epsilon}a^{\epsilon}\sqrt {a}\sum_{k_1^2\ll b-a}k_1\sum_{
a\le t_1<b,k_1^2|t_1^2-4}1.\]
For every given $k_1$ the inner sum is clearly $\ll_{\epsilon}a^{
\epsilon}\frac {b-a}{k_1^2}$,
because the number of solutions of the congruence $t_1^2\equiv 4$
modulo $k_1^2$ is $\ll_{\epsilon}a^{\epsilon}$. The lemma is proved.

\end{proof}

\subsection{Some notation and preliminaries.}

Throughout the paper we denote by $X$ and $d$ large enough
positive numbers such that $X^{2/3}\le d\le X^{99/100}$.

First we recall those notation from \cite{Biro} which will be used frequently in te sequel.

We denote by $X$ a large enough real number. We take integers $1\le
I\ll\log X$ and
$$3=a_1<a_2<\ldots <a_I<1+\sqrt {X+2}\le a_{I+1}<2+\sqrt {X+2}$$
such that
\[a_{i+1}\le\frac 32a_i,\;3+\sqrt {X+2}-a_i\le 2\left(3+\sqrt {X+
2}-a_{i+1}\right)
\]
for $1\le i\le I$.
We write
\begin{equation}a_{j_1,j_2}:=\left(-1\right)^{j_1+j_2}\left(\begin{array}{c}
J\\
j_1\end{array}
\right)\left(\begin{array}{c}
J\\
j_2\end{array}
\right).\label{888}\end{equation}

For $S,T>0$ we write
$$A\left(S,T\right):=\sqrt {\left(1+S^2\right)\left(1+T^2\right
)}-ST,$$
\[B\left(S,T\right):=\sqrt {\left(1+S^2\right)\left(1+T^2\right)}
+ST.\]
For a given $1\le\tau\le 2$ and for integers $t_1>2$, $t_2>2$ and $
f$ we write, assuming that we have nonnegative numbers under the square root,
\[S_0=S_0\left(j_1,t_1\right)=\sqrt {\frac {X-j_1d\tau -2}{t_1^2-
4}-1},\]
\[T_0=T_0\left(j_2,t_2\right)=\sqrt {\frac {X-j_2d\tau -2}{t_2^2-
4}-1}.\]
We introduce the convention that when we write simply $S_0$, then we mean
$S_0\left(j_1,t_1\right)$, and if we write simply $T_0$, then we mean $
T_0\left(j_2,t_2\right)$.

It will be important that in view of the identity
\[A\left(S_0,T_0\right)-1=\frac {\left(S_0-T_0\right)^2}{\sqrt {\left
(1+S_0^2\right)\left(1+T_0^2\right)}+S_0T_0+1}\]
(see the displayed formula below \cite[(6.49)]{Biro}) it is
true that for any $\epsilon >0$ there is an $\alpha >0$ such that if $
t_i,\sqrt X-t_i,\left|t_1-t_2\right|\gg X^{\frac 12-\alpha}$ for $
i=1,2$, then
\begin{equation}A\left(S_0,T_0\right)-1\gg X^{-\epsilon}.\label{741}\end{equation}
We write also
\begin{equation}F=F\left(t_1,t_2,f\right)=\left|\frac f{\sqrt {t_
1^2-4}\sqrt {t_2^2-4}}\right|.\label{742}\end{equation}
Let $s$ be a positive integer with $s\equiv 0,1$
modulo $4$ and define
$$\mathcal{Q}_s:=\left\{Q\left(X,Y\right)=AX^2+BXY+CY^2:\;A,B,C\in {\mathbb Z},\;
B^2-4AC=s\right\}.$$
If $\tau =\begin{pmatrix}a&b\\ c&d\end{pmatrix}\in SL_2({\mathbb Z})$ and $Q$ is a quadratic form, let us
define the quadratic form $Q^{\tau}$ by $Q^{\tau}\left(X,Y\right)
=Q\left(aX+bY,cX+dY\right)$. The
group $SL_2({\mathbb Z})$ acts in this way on $\mathcal{Q}_s$.

For $d_1,d_2,t\in {\mathbb Z}$, let ${\mathcal Q}_{d_1,d_2,t}$ be the subset of $\mathcal{Q}_{d_1}\times\mathcal{Q}_{d_2}$ consisting of those pairs $\left
(Q_1,Q_2\right)$ of quadratic forms having codiscriminant $t$. In other words, writing
$$Q_1\left(X,Y\right)=A_1X^2+B_1XY+C_1Y^2,\;Q_2\left(X,Y\right)=A_
2X^2+B_2XY+C_2Y^2$$
we require that the discriminant of $Q_j$ is $d_j$ ($j=1,2$) and that
$$B_1B_2-2A_1C_2-2A_2C_1=t.$$
It is easy to check that if $\tau\in SL_2({\mathbb Z})$, and
$\left(Q_1,Q_2\right)\in {\mathcal Q}_{d_1,d_2,t}$, then $\left(Q_1^{\tau},Q_
2^{\tau}\right)\in {\mathcal Q}_{d_1,d_2,t}$. Hence $SL_2({\mathbb Z})$
acts on $\mathcal Q_{d_1,d_2,t}$. We denote by $h\left(d_1,d_2,t\right
)$ the number of $SL_2(\mathbb Z)$-equivalence classes of $\mathcal
Q_{d_1,d_2,t}$.

We will frequently (and sometimes tacitly) use the
upper bound given for these class numbers proved in
\cite[Lemma 3.1]{Biro}.

We introduce also some new notation.

Let $\psi_0$ be a smooth function on the real axis
such that $\psi_0\left(y\right)=1$ for $y\ge 1$ and $\psi_0\left(
y\right)=0$ for $y\le 1/2$.

If $\beta_1>0$, let $f_{\beta_1}\left(y\right)$ be a real-valued smooth function defined for real $
y$ with the following properties. We have
\[f_{\beta_1}\left(y\right)=0\:\]
for $y\notin\left[X^{-\beta_1},1-X^{-\beta_1}\right]$,
\[f_{\beta_1}\left(y\right)=1\:\]
for $y\in\left[2X^{-\beta_1},1-2X^{-\beta_1}\right]$, and
\[f_{\beta_1}^{\left(r\right)}\left(y\right)\ll_rX^{\beta_1r}\]
for every $y>0$ and every integer $r\ge 0$.

If $c$ is an odd integer, let $\epsilon_c=1$ if $c\equiv 1$ modulo $
4$, and $\epsilon_c=i$ if $c\equiv -1$ modulo $4$.

For the definition of the Hilbert symbol $\left(a,b\right)_p$ and the
quantities $\nu_p\left(x\right)$ and $h_p\left(d_1,d_2,t\right)$ see the Introduction of \cite{biroClass}.

\subsection{Identifying the critical parts of the summation over $
t_1,t_2,f$.}

Recall the notation $y_1$, $y_2$ and $\Phi(y)$ from \cite[Lemma 4.3]{Biro}.

\begin{lemma}\label{lem:2} Let $\epsilon >0$ be given and let
\begin{equation}X^{2/3}\le d\le X^{99/100}.\label{6,2}\end{equation}
There is a $0<\delta_0<\frac 1{100}$ which is fixed in terms of $
\epsilon$ such that for any
$\delta >0$ which is small enough in terms of $\epsilon$ and for any
positive integer $J$  which is large enough in terms of $\delta$
and $\epsilon$ we have for every $1\le\tau\le 2$ that $I_{d,X,J}\left
(\tau\right)$ is $\ll_{\epsilon ,\delta ,J}$ than the sum of
\begin{equation}\sum_{i=1}^{I_0}\sum_{j_1=0}^J\sum_{j_2=0}^Ja_{j_
1,j_2}\sum_{\left(t_1,t_2,f\right)\in U_{i,j_1,j_2}}h\left(t_1^2-
4,t_2^2-4,f\right)S_0\Phi\left(y_1\left(S_0,T_0,F\right)\right),\label{6,3}\end{equation}
\begin{equation}\sum_{i=1}^{I_0}\sum_{j_1=0}^J\sum_{j_2=0}^Ja_{j_
1,j_2}\sum_{\left(t_1,t_2,f\right)\in V_{i,j_1,j_2}}h\left(t_1^2-
4,t_2^2-4,f\right)S_0\Phi\left(y_2\left(S_0,T_0,F\right)\right),\label{6,4}\end{equation}
\begin{equation}-\sum_{i=1}^{I_0}\sum_{j_1=0}^J\sum_{j_2=0}^Ja_{j_
1,j_2}\sum_{\left(t_1,t_2,f\right)\in W_{i,j_1,j_2}}h\left(t_1^2-
4,t_2^2-4,f\right)S_0\Phi\left(y_2\left(S_0,T_0,F\right)\right)\label{6,5}\end{equation}
and
\begin{equation}X^{\epsilon}\left(\frac {d^3}X+\frac {d^2}{X^{1/4}}
+\sqrt {X}d\right),\label{7}\end{equation}
where $I_0$ is the maximal integer $i$ satisfying $1\le i\le I$
and
\begin{equation}\sqrt {X+2}-a_i>\frac d{\sqrt {X}}X^{\delta_0},\label{10}\end{equation}
$U_{i,j_1,j_2}$ is the set of integer triples $\left(t_1,t_2,f\right
)$ satisfying the following
conditions:
\begin{equation}B\left(S_0\left(j_1,t_1\right),T_0\left(0,t_2\right
)\right)-\frac d{a_i^2}X^{\delta}<F\le B\left(S_0\left(j_1,t_1\right
),T_0\left(j_2,t_2\right)\right),\label{7,9}\end{equation}
\begin{equation}F>1,\label{8}\end{equation}
\begin{equation}a_i\le t_1,t_2\le a_{i+1}-1,\label{9}\end{equation}
$V_{i,j_1,j_2}$ is the set of integer triples $\left(t_1,t_2,f\right
)$ satisfying
\eqref{8}, \eqref{9} and the following conditions:
\begin{equation}A\left(S_0\left(j_1,t_1\right),T_0\left(0,t_2\right
)\right)-\frac {d\left|t_2-t_1\right|}{a_i\left(X-a_i^2\right)}X^{
\delta}<F\le A\left(S_0\left(j_1,t_1\right),T_0\left(j_2,t_2\right
)\right),\label{11}\end{equation}
\begin{equation}t_1-t_2>\frac {da_i}XX^{\delta_0},\label{11,1}\end{equation}
$W_{i,j_1,j_2}$ is the set of integer triples $\left(t_1,t_2,f\right
)$ satisfying
\eqref{8}, \eqref{9}, \eqref{11} and the following condition:
\begin{equation}t_2-t_1>\frac {da_i}XX^{\delta_0}.\label{11,2}\end{equation}
\end{lemma}

\begin{remark}Note that $U_{i,j_1,j_2}$, $V_{i,j_1,j_2}$ and $W_{
i,j_1,j_2}$
depend also on $\delta_0$ and $\delta$, but we do not denote it. Note
also that the middle term in \eqref{7} could be omitted as
it is the geometric mean of the other two terms,
but we keep it since it will be convenient in the course
of the proof.
\end{remark}

\begin{proof}This lemma can be proved by the reasoning
of \cite[Section 6]{Biro}, using at a
few places Lemma \ref{lem:1} above instead of
\cite[(3.20)]{Biro}.

In more detail, our reasoning is the following. First we see by
\cite[(6.3)-(6.7)]{Biro} that it is enough to
estimate \cite[(6.5), (6.6) and (6.7)]{Biro}.

Then, as in \cite[Subsection 6.2]{Biro}, we consider the contribution to
\cite[(6.5), (6.6) and (6.7)]{Biro} of those $i$ for which we have
\begin{equation}\sqrt {X+2}-a_i=O\left(\frac d{\sqrt {X}}X^{\delta}\right
)\label{118}\end{equation}
for some $\delta >0$ which is chosen small enough in terms of
$\epsilon$. We prove that this contribution can be estimated by
\eqref{7}. In the case of \cite[(6.7) and (6.5)]{Biro} it is proved in
\cite[Subsection 6.2]{Biro}. In the case of \cite[(6.6)]{Biro}
we apply Lemma \ref{lem:1} above instead of
\cite[(3.20)]{Biro} in the reasoning below \cite[(6.22)]{Biro}.
We get in this way, using also the condition \eqref{118}
above that the $t_1\neq t_2$ part of
\cite[(6.22)]{Biro} is
\[\ll_{\delta}X^{3\delta}d\frac d{\sqrt {X}}X^{\delta}\left(\frac
d{\sqrt {X}}X^{\delta}+\sqrt {a_i}\right),\]
which is smaller than \eqref{7}, if $\delta$ is small
enough. Hence we can estimate by \eqref{7} also the contribution to \cite[(6.6)]{Biro} of those $
i$ for which
we have \eqref{118}. Hence we can assume from now on in \cite[(6.5), (6.6) and (6.7)]{Biro} that \eqref{10} is true with some $
\delta_0>0$ which is fixed in terms of $\epsilon$.

It is then proved in \cite[Subsection 6.4]{Biro} that
\cite[(6.5)]{Biro} can be estimated by $X^{\epsilon}\sqrt {X}d$. So it is
enough to consider \cite[(6.6) and (6.7)]{Biro}.

Next, as in \cite[Subsection 6.5]{Biro}, we consider the
contribution of \cite[(6.6)]{Biro}. In the reasoning below
\cite[(6.39)]{Biro} we apply again Lemma \ref{lem:1} above instead of
\cite[(3.20)]{Biro}. We get in this way that that part of
\cite[(6.6)]{Biro} where \cite[(6.36) and (6.37)]{Biro} and
also $t_1\neq t_2$ hold is
\[\ll_{\delta}X^{4\delta}\left(\frac d{a_i}+\frac {d^2}{X-a_i^2}\right
)\left(\left(\min\left(a_i,\sqrt {X}-a_i\right)\right)\left(\frac {
da_i}X+\sqrt {a_i}\right)\right)\]
for any $\delta >0$, and this can be easily estimated by the
first two terms of \eqref{7} if $\delta$ is small enough. Using
the rest of the reasoning of \cite[Subsection 6.5]{Biro}
without any change we see that \cite[(6.6)]{Biro} can be
estimated by \eqref{7}. So it is enough to consider
\cite[(6.7)]{Biro}. In particular, we can assume from now
on that \eqref{8} is true.

We now use the new expression \cite[(6.45)]{Biro} for \cite[(6.7)]{Biro}
given in \cite[Subsection 6.6]{Biro}.

We first consider the contribution of $B_{t_1,t_2}$ and $C_{t_1,t_
2}$
(using the notations introduced in \cite[Subsection
6.6]{Biro}) to \cite[(6.7)]{Biro}. Note that in the definition
of  $B_{t_1,t_2}$ and $C_{t_1,t_2}$ (see \cite[(6.47) and (6.48)]{Biro}) we
have
\[1<F\le A\left(S_0,T_0\right).\]
Estimating the second displayed formula below \cite[(6.49)]{Biro} by Lemma \ref{lem:1} above instead of
\cite[(3.20)]{Biro} we get that the contribution to
\cite[(6.7)]{Biro} of the terms $B_{t_1,t_2}$, $C_{t_1,t_2}$ satisfying
\[\left|t_2-t_1\right|\ll\frac {da_i}XX^{\delta}\]
for some $\delta >0$ is
\[\ll_{\delta}X^{5\delta}\left(\frac {\sqrt {X-a^2_i}}{a_i}+\frac {
d^2a_i/X}{\sqrt {X-a^2_i}}\right)P\]
with
\[P:=\left(a_i\sqrt {\sqrt {X}-a_i}+\left(\sqrt {X}-a_i\right)\left
(a_i\frac dX+\sqrt {a_i}\right)\right).\]
One can see using $d\ge X^{2/3}$ that this can be estimated by the
first two terms of \eqref{7} if $\delta$ is small enough. Hence
when we consider the contribution of $B_{t_1,t_2}$ and $C_{t_1,t_
2}$ to
\cite[(6.7)]{Biro} we can assume from now on that
\begin{equation}\left|t_2-t_1\right|>\frac {da_i}XX^{\delta_0}\label{13}\end{equation}
with some $\delta_0>0$ which is fixed in terms of $\epsilon$. Note also
that it is proved in the beginning of \cite[Subsection
6.4]{Biro} that if \eqref{10} and \eqref{13} are true, then the sign of
$S_0\left(j_1,t_1\right)-T_0\left(j_2,t_2\right)$ is the same for every pair
$0\le j_1,j_2\le J$. And for $j_1=j_2=0$ we see by the
definitions that this sign is the same as the sign of
$t_2-t_1$.

It is proved in \cite[Subsection 6.7]{Biro} that the contribution of those terms
$B_{t_1,t_2}$$,C_{t_1,t_2}$ to \cite[(6.7)]{Biro} for which
\[1<F\le A\left(S_0\left(j_1,t_1\right),T_0\left(0,t_2\right)\right
)-\frac {d\left|t_2-t_1\right|}{a_i\left(X-a_i^2\right)}X^{\delta}\]
holds for some $\delta >0$ is negligibly small. Hence we get
that the contribution of $B_{t_1,t_2}$ and $C_{t_1,t_2}$ to
\cite[(6.7)]{Biro} can be estimated by the sum of
\eqref{6,4}, \eqref{6,5} and \eqref{7}.

Finally, in \cite[Subsection 6.8]{Biro} it is proved that the contribution of those terms
$A_{t_1,t_2}$ to \cite[(6.7)]{Biro} for which
\[1<F\le B\left(S_0\left(j_1,t_1\right),T_0\left(0,t_2\right)\right
)-\frac d{a_i^2}X^{\delta}\]
holds for some $\delta >0$ is negligibly small.
Hence we get that the contribution of $A_{t_1,t_2}$ to
\cite[(6.7)]{Biro} can be estimated by the sum of
\eqref{6,3} and \eqref{7}. The lemma is proved.

\end{proof}

\subsection{Further reductions of the summation over $t_1,t_2,f$.}

\begin{lemma}\label{lem:3}Assume that \eqref{6,2} holds. For $1\le
i\le I_0$ and for
$0\le j_1,j_2\le J$ let us write $Z_{i,j_1,j_2}:=V_{i,j_1,j_2}\cup
W_{i,j_1,j_2}$ and denote by $S_{i,j_1,j_2}$ the following sum:
\[\sum_{\left(t_1,t_2,f\right)\in Z_{i,j_1,j_2}}\operatorname{sgn}
(t_1-t_2)h\left(t_1^2-4,t_2^2-4,f\right)S_0\Phi\left(y_2\left(S_0
,T_0,F\right)\right).\]
Let $\alpha >0$ and let the parameter $\delta$ used in Lemma \ref{lem:2} be small enough in terms of $\alpha$. Then there is a $\beta >0$ such that the following
statements are true.
If $\min\left(a_i,\sqrt {X}-a_i\right)<X^{\frac 12-\alpha}$, then
\begin{equation}S_{i,j_1,j_2}\ll\frac {d^{5/2}}{\sqrt {X}}X^{-\beta}
.\label{14}\end{equation}
If $\min\left(a_i,\sqrt {X}-a_i\right)\ge X^{\frac 12-\alpha}$, then
\begin{equation}\sum_{\left(t_1,t_2,f\right)\in Z_{i,j_1,j_2},\left
|t_1-t_2\right|<X^{\frac 12-\alpha}}\left|h\left(t_1^2-4,t_2^2-4,
f\right)S_0\Phi\left(y_2\left(S_0,T_0,F\right)\right)\right|\label{15}\end{equation}
is $\ll\frac {d^{5/2}}{\sqrt {X}}X^{-\beta}$, and also
\begin{equation}\sum_{\left(t_1,t_2,f\right)\in Z_{i,j_1,j_2},S\left
(t_1^2-4,t_2^2-4,f^2\right)>X^{\alpha}}\left|h\left(t_1^2-4,t_2^2
-4,f\right)S_0\Phi\left(y_2\left(S_0,T_0,F\right)\right)\right|\label{16}\end{equation}
is $\ll\frac {d^{5/2}}{\sqrt {X}}X^{-\beta}$.
\end{lemma}

\begin{proof}We will follow closely the arguments of \cite[Subsection
6.7]{Biro}.

The following facts are proved in \cite[Subsection
6.7]{Biro} assuming that \eqref{11,1} or \eqref{11,2} is true with some $
\delta_0\ge\delta$:

{\bf FACT 1.} The number of possible values of the integers $f$ satisfying \eqref{11} is
\begin{equation}\ll 1+\frac {d\left|t_2-t_1\right|a_i}{X-a_i^2}X^{
\delta}.\label{6.53}\end{equation}
{\bf FACT 2.} Assume that \eqref{11} holds. If $a_i\ge\frac {\sqrt {
X}}2$, then
\begin{equation}S_0\Phi\left(y_2\left(S_0,T_0,F\right)\right)\ll_{
\delta}X^{4\delta}\sqrt {\frac {d^3\left(\sqrt {X}-a_i\right)}{X^
2\left|t_2-t_1\right|^3}}.\label{6.55}\end{equation}
{\bf FACT 3.} Assume that \eqref{11} holds. If $a_i\le\frac {\sqrt {
X}}2$ and $\left|t_2-t_1\right|\ll\frac d{a_i}$, then
\begin{equation}S_0\Phi\left(y_2\left(S_0,T_0,F\right)\right)\ll_{
\delta}X^{2\delta}\sqrt {\frac d{a_i\left|t_2-t_1\right|}}.\label{6.56}\end{equation}
{\bf FACT 4.} Assume that \eqref{11} holds. If $a_i\le\frac {\sqrt {
X}}2$ and $\left|t_2-t_1\right|\gg\frac d{a_i}$, then
\begin{equation}S_0\Phi\left(y_2\left(S_0,T_0,F\right)\right)\ll_{
\delta}X^{4\delta}\sqrt {\frac {d^3}{a_i^3\left|t_2-t_1\right|^3}}
.\label{6.57}\end{equation}
If $a_i\ge\frac {\sqrt {X}}2$, then by \eqref{11,1}, \eqref{11,2} and  \eqref{6,2}
we see that the second term is larger than the first one
in \eqref{6.53}. Then by \eqref{6.53}, \eqref{6.55} and
\cite[Lemma 3.1]{Biro} we see for any $R>0$ that
\begin{equation}\sum_{\left(t_1,t_2,f\right)\in Z_{i,j_1,j_2},\left
|t_1-t_2\right|<R}\left|h\left(t_1^2-4,t_2^2-4,f\right)S_0\Phi\left
(y_2\left(S_0,T_0,F\right)\right)\right|\end{equation}
is
\[\ll_{\delta}X^{6\delta}\frac {d^{5/2}}{X^{3/4}\sqrt {X-a_i^2}}\sum_{
t_1=a_i}^{a_{i+1}-1}\sum_{a_i\le t_2<a_{i+1},\frac {da_i}XX^{\delta_
0}<\left|t_2-t_1\right|<R}\frac {S\left(t^2_1-4,t^2_2-4\right)}{\sqrt {\left
|t_2-t_1\right|}}.\]
For $R=a_{i+1}-a_i$ we have by \cite[(3.21)]{Biro} that the sum
over $t_1$,$t_2$ is $\ll_{\delta}X^{\delta}\left(\sqrt {X}-a_i\right
)X^{1/4}$, and this shows
\eqref{14} for the case $a_i\ge\frac {\sqrt {X}}2$. Taking $R=x^{\frac
12-\alpha}$ we
obtain \eqref{15} for the case $a_i\ge\frac {\sqrt {X}}2$ by a dyadic
subdivision and by Lemma \ref{lem:1} above.

If $a_i\le\frac {\sqrt {X}}2$, then by \eqref{6.53} and
\eqref{6.56} we see that
\begin{equation}\sum_{\left(t_1,t_2,f\right)\in Z_{i,j_1,j_2},\left
|t_1-t_2\right|\ll\frac d{a_i}}\left|h\left(t_1^2-4,t_2^2-4,f\right
)S_0\Phi\left(y_2\left(S_0,T_0,F\right)\right)\right|\end{equation}
is $\ll_{\delta}$ than the sum of
\begin{equation}X^{4\delta}\frac {d^2}X\sum_{t_1=a_i}^{a_{i+1}-1}
\sum_{a_i\le t_2<a_{i+1},0<\left|t_2-t_1\right|\ll\frac d{a_i}}S\left
(t^2_1-4,t^2_2-4\right)\label{89}\end{equation}
and
\begin{equation}X^{4\delta}\sqrt {\frac d{a_i}}\sum_{t_1=a_i}^{a_{
i+1}-1}\sum_{a_i\le t_2<a_{i+1},0<\left|t_2-t_1\right|}\frac {S\left
(t^2_1-4,t^2_2-4\right)}{\sqrt {\left|t_2-t_1\right|}}.\label{90}\end{equation}
By Lemma \ref{lem:1} we have that \eqref{89} is
$\ll_{\delta}X^{5\delta}\frac {d^2a_i}X\left(\frac d{a_i}+\sqrt {
a_i}\right)$, and this is smaller than $\ll\frac {d^{5/2}}{\sqrt {
X}}X^{-\beta}$
with a $\beta >0$ if $\delta$ is small enough. We get the same
conclusion for \eqref{90} in the following way: we estimate \eqref{90} by
\cite[(3.21)]{Biro}, and the result is an upper bound
$X^{5\delta}\left(d/a_i\right)^{1/2}a_i^{3/2}\ll X^{5\delta}\sqrt {
d}\sqrt {X}$, and this is really smaller than $\ll\frac {d^{5/2}}{\sqrt {
X}}X^{-\beta}$
with a $\beta >0$ if $\delta$ is small enough.

If $a_i\le\frac {\sqrt {X}}2$ and $\left|t_2-t_1\right|\gg\frac d{
a_i}$, then by \eqref{6,2} we see that
the second term in \eqref{6.53} is larger than the first one.
Hence by \eqref{6.53} and \eqref{6.57} we see in the case
$a_i\le\frac {\sqrt {X}}2$ for any $R>0$ that
\begin{equation}\sum_{\left(t_1,t_2,f\right)\in Z_{i,j_1,j_2},\frac
d{a_i}\ll\left|t_1-t_2\right|<R}\left|h\left(t_1^2-4,t_2^2-4,f\right
)S_0\Phi\left(y_2\left(S_0,T_0,F\right)\right)\right|\end{equation}
is
\[\ll_{\delta}X^{6\delta}\frac {d^{5/2}}{a_i^{1/2}X}\sum_{t_1=a_i}^{
a_{i+1}-1}\sum_{a_i\le t_2<a_{i+1},\frac d{a_i}\ll\left|t_2-t_1\right
|<R}\frac {S\left(t^2_1-4,t^2_2-4\right)}{\sqrt {\left|t_2-t_1\right
|}}.\]
For $R=a_{i+1}-a_i$ we have by \cite[(3.21)]{Biro} that this
is $\ll_{\delta}X^{7\delta}\frac {d^{5/2}a_i}X$, and this shows \eqref{14} also for the
case $a_i\le\frac {\sqrt {X}}2$. Taking $R=X^{\frac 12-\alpha}$ we
obtain \eqref{15} for the case $a_i\le\frac {\sqrt {X}}2$ by a dyadic
subdivision and by Lemma \ref{lem:1} above.

Hence we proved \eqref{14} and \eqref{15} completely.

Then for the proof of \eqref{16} we can assume
$\min\left(a_i,\sqrt {X}-a_i\right)\ge x^{\frac 12-\epsilon}$, and it is enough to consider the
part of the summation where $\left|t_1-t_2\right|\ge X^{\frac 12-
\epsilon}$ (here $\epsilon >0$
ia an arbitrary fixed number). Then by \eqref{6.55} and
\eqref{6.57} we see that
\begin{equation}S_0\Phi\left(y_2\left(S_0,T_0,F\right)\right)\ll_{
\delta}X^{5\delta}\frac {d^{3/2}}{X^{3/2}}.\label{20}\end{equation}
We also see from \eqref{11} and \cite[(6.29)]{Biro} that for given $
t_1$ and $t_2$ the
number of possible integers $f$ is $\ll_{\delta}dX^{2\delta}$.

Now let $k:=S\left(t_1^2-4,t_2^2-4,f^2\right)$. Then $k^2$ divides $
t_1^2-4$,
$t_2^2-4$ and $f^2$, hence $k$ divides $t_1^2-4$,
$t_2^2-4$ and $f$. We also have that $k^2\le t_1^2-4\ll X$. Hence the
number of possible integers $t_1$, $t_2$ and $f$ divisible by $k$ is
$\ll_{\delta}\sqrt XX^{\delta}/k$, $\ll_{\delta}\sqrt XX^{\delta}
/k$ and $\ll_{\delta}dX^{3\delta}/k$, respectively. Using
now \eqref{20} and trivial estimates we obtain
\eqref{16}. the lemma is proved.

\end{proof}

\begin{lemma}\label{lem:4}Assume that \eqref{6,2} holds. For $1\le
i\le I_0$ and for
$0\le j_1,j_2\le J$ let us denote by $T_{i,j_1,j_2}$ the following sum:
\[\sum_{\left(t_1,t_2,f\right)\in U_{i,j_1,j_2}}h\left(t_1^2-4,t_
2^2-4,f\right)S_0\Phi\left(y_1\left(S_0,T_0,F\right)\right).\]
Let $\alpha >0$ and let the parameter $\delta$ used in Lemma \ref{lem:2} be small enough in terms of $\alpha$. Then there is a $\beta >0$ such that the following
statements are true.
If $\min\left(a_i,\sqrt {X}-a_i\right)<X^{\frac 12-\alpha}$, then
\begin{equation}T_{i,j_1,j_2}\ll\frac {d^{5/2}}{\sqrt {X}}X^{-\beta}
.\label{34}\end{equation}
If $\min\left(a_i,\sqrt {X}-a_i\right)\ge X^{\frac 12-\alpha}$, then
\begin{equation}\sum_{\left(t_1,t_2,f\right)\in U_{i,j_1,j_2},S\left
(t_1^2-4,t_2^2-4,f^2\right)>X^{\alpha}}\left|h\left(t_1^2-4,t_2^2
-4,f\right)S_0\Phi\left(y_1\left(S_0,T_0,F\right)\right)\right|\label{36}\end{equation}
is $\ll\frac {d^{5/2}}{\sqrt {X}}X^{-\beta}$.
\end{lemma}

\begin{proof}It is proved in \cite[Subsection
6.8]{Biro} assuming that \eqref{7,9} holds that
\begin{equation}S_0\Phi\left(S_0,y_1\left(S_0,T_0,F\right)\right)
\ll_{\delta}X^{3\delta}\frac {d^{3/2}}{X\left(\sqrt {X}-a_i\right
)},\label{92}\end{equation}
and for given $t_1$ and $t_2$ the number of possible values of the integers $
f$
satisfying \eqref{7,9} is $\ll_{\delta}X^{\delta}d$. Therefore $T_{
i,j_1,j_2}$ is
\begin{equation}\ll_{\delta}X^{5\delta}\frac {d^{5/2}}{X\left(\sqrt {
X}-a_i\right)}\sum_{t_1=a_i}^{a_{i+1}-1}\sum_{t_2=a_i}^{a_{i+1}-1}
S\left(t^2_1-4,t^2_2-4\right).\label{93}\end{equation}
We estimate the $t_1=t_2$ part here by \cite[(3.22)]{Biro}, and
we get by \eqref{10} and \eqref{6,2} that this part gives
$\ll\frac {d^{5/2}}{\sqrt {X}}X^{-\beta}$ for some $\beta >0$ if $
\delta$ is small enough. The $t_1\neq t_2$
part of the sum here is estimated by Lemma \ref{lem:1}
above. we get in this way that the $t_1\neq t_2$ part of
\eqref{93} is
\[\ll_{\delta}X^{5\delta}\frac {d^{5/2}}{X\left(\sqrt {X}-a_i\right
)}\left(\sqrt {a_i}+\min\left(\sqrt {X}-a_i,a_i\right)\right)\min\left
(\sqrt {X}-a_i,a_i\right).\]
This proves \eqref{34}. For the proof of \eqref{36} we
use \eqref{92} instead of \eqref{20}, otherwise the proof is completely similar
to the proof of \eqref{16}. The lemma is proved.
\end{proof}

\subsection{Asymptotics for $S_0\Phi\left(y_j\left(S_0,T_0,F\right)\right)$.}

\begin{lemma}\label{lem:5}Let $1\le i\le I_0$ and $0\le j_1,j_2\le
J$,
and let $t_1,t_2$ satisfy \eqref{9}. Write $S_0=S_0\left(j_1,t_1\right
)$,
$T_0=T_0\left(j_2,t_2\right)$. Assume \eqref{6,2},
\begin{equation}B\left(S_0\left(j_1,t_1\right),T_0\left(j_2,t_2\right
)\right)-\frac 1{X^{1/1000}}\le F\le B\left(S_0\left(j_1,t_1\right
),T_0\left(j_2,t_2\right)\right)\label{78,78}\end{equation}
and let
\begin{equation}\min\left(a_i,\sqrt {X}-a_i\right)\ge X^{\frac 12
-\delta}\label{79}\end{equation}
hold with some $\delta >0$. If $\delta$ is small enough, then we have that
\[S_0\Phi\left(y_1\left(S_0,T_0,F\right)\right)=A_X\left(t_1,t_2\right
)\left(B\left(S_0,T_0\right)-F\right)^{3/2}+O\left(\frac {d^{3/2}}{
X^{3/2}}X^{-\beta_0}\right)\]
with an absolute constant $\beta_0>0$, where
\begin{equation}A_X\left(t_1,t_2\right):=\frac {\sqrt {X-t_1^2}}{
t_1}\frac X{3t_1^2}\left(\frac {\left(\frac {2X}{t_1t_2}\right)^{
1/2}}{\frac {\sqrt {X-t_2^2}}{t_2}+\frac {\sqrt {X-t_1^2}}{t_1}\left
(\frac {X+\sqrt {X-t_1^2}\sqrt {X-t_2^2}}{t_1t_2}\right)}\right)^
3.\label{79,1}\end{equation}
\end{lemma}
\begin{proof}It is easy to compute (see \cite[(6.65)]{Biro}) that
\[y_1\left(S_0,T_0,F\right)=\sqrt {\frac {\left(B\left(S_0,T_0\right
)-F\right)\left(A\left(S_0,T_0\right)+F\right)}{\left(T_0+S_0F\right
)^2}},\]
and then it follows from the definition of $S_0$ and $T_0$,
using \eqref{6,2}, \eqref{78,78} and \eqref{79} with a small enough $
\delta >0$ that $y_1\left(S_0,T_0,F\right)$ equals
\[\sqrt {B\left(S_0,T_0\right)-F}\frac {\left(\frac {2X}{t_1t_2}\right
)^{1/2}}{\frac {\sqrt {X-t_2^2}}{t_2}+\frac {\sqrt {X-t_1^2}}{t_1}\left
(\frac {X+\sqrt {X-t_1^2}\sqrt {X-t_2^2}}{t_1t_2}\right)}\left(1+
O\left(X^{-\beta_1}\right)\right)\]
with an absolute constant $\beta_1>0$. This shows also under the
same conditions that $y_1\left(S_0,T_0,F\right)=O\left(X^{-\beta_
2}\right)$ with an
absolute constant $\beta_2>0$. It is easy to see by the
definition of the function $\Phi$ in \cite[Lemma 4.3]{Biro} that
\[\Phi\left(y_1\left(S_0,T_0,F\right)\right)=\left(1+S_0^2\right)\frac {\left
(y_1\left(S_0,T_0,F\right)\right)^3}3\left(1+O\left(X^{-\beta_3}\right
)\right)\]
with an absolute constant $\beta_3>0$. The lemma follows.
\end{proof}

\begin{lemma}\label{lem:6}Let $1\le i\le I_0$ and $0\le j_1,j_2\le
J$,
and let $t_1,t_2$ satisfy \eqref{9}. Write $S_0=S_0\left(j_1,t_1\right
)$,
$T_0=T_0\left(j_2,t_2\right)$. Assume \eqref{6,2},
\begin{equation}A\left(S_0\left(j_1,t_1\right),T_0\left(j_2,t_2\right
)\right)-\frac 1{X^{1/1000}}\le F\le A\left(S_0\left(j_1,t_1\right
),T_0\left(j_2,t_2\right)\right)\label{78,780}\end{equation}
and let \eqref{79} and
\begin{equation}\left|t_1-t_2\right|\ge X^{\frac 12-\delta}\label{79,15}\end{equation}
hold with some $\delta >0$. If $\delta$ is small enough, then we have that
\[S_0\Phi\left(y_2\left(S_0,T_0,F\right)\right)=B_X\left(t_1,t_2\right
)\left(A\left(S_0,T_0\right)-F\right)^{3/2}+O\left(\frac {d^{3/2}}{
X^{3/2}}X^{-\beta_0}\right)\]
with an absolute constant $\beta_0>0$, where
\begin{equation}B_X\left(t_1,t_2\right):=\frac {X\sqrt {X-t_1^2}}
3\left(\frac {\left(2t_1t_2\right)^{1/2}\left(\sqrt {X-t_1^2}+\sqrt {
X-t_2^2}\right)}{\left|t_1^2-t_2^2\right|\sqrt {X}}\right)^3.\label{79,2}\end{equation}
\end{lemma}
\begin{proof}It is easy to compute (see the displayed
formula below \cite[(6.53)]{Biro}) that
\[y_2\left(S_0,T_0,F\right)=\sqrt {\frac {\left(B\left(S_0,T_0\right
)+F\right)\left(A\left(S_0,T_0\right)-F\right)}{\left(T_0-S_0F\right
)^2}},\]
and then it follows from the definition of $S_0$ and $T_0$,
\cite[(6.30)]{Biro}, using \eqref{6,2}, \eqref{78,780},
\eqref{79} with a small enough $\delta >0$ that $y_2\left(S_0,T_0
,F\right)$ equals
\[\sqrt {A\left(S_0,T_0\right)-F}\frac {t_1\left(2t_1t_2\right)^{
1/2}\left(\sqrt {X-t_1^2}+\sqrt {X-t_2^2}\right)}{\left|t_1^2-t_2^
2\right|\sqrt {X}}\left(1+O\left(X^{-\beta_1}\right)\right)\]
with an absolute constant $\beta_1>0$. This shows also under the
same conditions that $y_2\left(S_0,T_0,F\right)=O\left(X^{-\beta_
2}\right)$ with an
absolute constant $\beta_2>0$. The end of the proof is the
same as the end of the proof of Lemma \ref{lem:5}.
\end{proof}

\subsection{Arithmetic reduction of the summation over $t_1,t_2,f$.}

\begin{lemma}\label{lem:7,5}
Let $y\ge 9$ and let $u,v>0$ be integers such that
$u^2>16y^2$ and $v\gg y^a$, $u+v\ll y^b$ with some absolute
constants $a,b>0$. For integers $t_1,t_2,f$ satisfying
$\sqrt {y}\le t_1,t_2\le 2\sqrt {y}$ and $u\le f<u+v$ let us write
\[E_{t_1,f}:=\gcd\left(t_1^2-4,f\right),\]
\begin{equation}G_{t_1,t_2,f}:=\prod_{p|2E_{t_1,f}}p^{\nu_p\left(
f^2-\left(t_1^2-4\right)\left(t_2^2-4\right)\right)},\label{222}\end{equation}
\[R_{t_1,f}:=\prod_{p|2E_{t_1,f}}p^{\nu_p\left(t_1^2-4\right)}.\]
Then for every $c>0$ there is a $d>0$ such that for any $t_1,t_2$ satisfying
$\sqrt {y}\le t_1,t_2\le 2\sqrt {y}$ we have
\begin{equation}\left|\left\{f\in\mathbb{Z}:u\le f<u+v,G_{t_1,t_2
,f}\ge y^c\right\}\right|\ll_{c,d}vy^{-d},\label{74}\end{equation}
and for any $f$ satisfying $u\le f<u+v$ we have
\begin{equation}\left|\left\{t_1\in\mathbb{Z}:\sqrt {y}\le t_1\le
2\sqrt {y},R_{t_1,f}\ge y^c\right\}\right|\ll_{c,d}y^{\frac 12-d}
.\label{75}\end{equation}
\end{lemma}
\begin{proof}If $t_1$ is fixed or if $f$ is fixed, then the number of
possible values of $E_{t_1,f}$ is  $\ll_{\epsilon}y^{\epsilon}$ for any $
\epsilon >0$. So during the proof of
either \eqref{74} or \eqref{75} we can fix $E=E_{t_1,f}$.

The number $\omega\left(E\right)$ of prime factors of $E$ is $\ll\log
y$ (by a very
crude estimate). Then by \cite[(1.14)]{Granville} we see that
\begin{equation}\left|\left\{n\in\mathbb{Z}:0<n\le\left(u+v\right
)^2+4y,p|n\implies p|E\right\}\right|\ll_{\epsilon}y^{\epsilon}\label{987}\end{equation}
for every $\epsilon >0$. We used here that the number of integers
up to $\left(u+v\right)^2+4y$ composed of
the powers of the first $\omega\left(E\right)$ primes is an upper
bound for the left-hand side of \eqref{987}.

Let $c>0$ and assume $G_{t_1,t_2,f}\ge y^c$ in the case of
\eqref{74} and $R_{t_1,f}\ge y^c$ in the case of \eqref{75}. These
numbers can take $\ll_{\epsilon}y^{\epsilon}$ possible values by \eqref{987}
for every $\epsilon >0$. So we can fix also $G=G_{t_1,t_2,f}$ in the case of
\eqref{74} and $R=R_{t_1,f}$ in the case of \eqref{75}.

In the case of \eqref{74} we have $G|f^2-\left(t_1^2-4\right)\left
(t_2^2-4\right)$
and $E|f$. Assume that $E\ge y^e$ with some $e>0$. Then the
possible values of $f$ are $\ll_{\epsilon}y^{\epsilon}\max\left(v
y^{-e},1\right)$ for every
$\epsilon >0$, which would give \eqref{74}. Hence we can assume
that $E\ll_{\delta}y^{\delta}$ for any $\delta >0$. By the definition \eqref{222}
of $G$ it follows then that there is a divisor $G_0$ of $G$
such that $G_0\ll v$ but $G_0\gg y^{c_0}$ with some $c_0>0$. Since
we have $G_0|f^2-\left(t_1^2-4\right)\left(t_2^2-4\right)$, we get that the possible
values of $f$ are $\ll_{\epsilon}y^{\epsilon}vy^{-c_0/2}$ for every $
\epsilon >0$. This proves \eqref{74}.

In the case of \eqref{75} we have $R|t_1^2-4$, hence
the possible values of $t_1$ are $\ll_{\epsilon}y^{\epsilon}\max\left
(y^{\frac 12-c},1\right)$ for every
$\epsilon >0$. This proves \eqref{75}.
\end{proof}

\begin{lemma}\label{lem:8}Assume that \eqref{6,2} holds. For $1\le
i\le I_0$ and for
$0\le j_1,j_2\le J$ let us write $Z_{i,j_1,j_2}:=V_{i,j_1,j_2}\cup
W_{i,j_1,j_2}$. Let
$\alpha >0$, let the parameter $\delta$ used in Lemma \ref{lem:2} be
small enough in terms of $\alpha$, and assume
$\min\left(a_i,\sqrt {X}-a_i\right)\ge X^{\frac 12-\alpha}$.

Let $G_{t_1,t_2,f}$ and $R_{t_1,f}$ be as in Lemma \ref{lem:7,5}, and
let us write $M_{t_1,t_2,f}:=\max\left(G_{t_1,t_2,f},R_{t_1,f}\right
)$. Then there is a $\beta >0$ such that
\begin{equation}\sum_{\left(t_1,t_2,f\right)\in Z_{i,j_1,j_2},M_{
t_1,t_2,f}>X^{\alpha}}\left|h\left(t_1^2-4,t_2^2-4,f\right)S_0\Phi\left
(y_2\left(S_0,T_0,F\right)\right)\right|\label{103}\end{equation}
is $\ll\frac {d^{5/2}}{\sqrt {X}}X^{-\beta}$ and also
\begin{equation}\sum_{\left(t_1,t_2,f\right)\in U_{i,j_1,j_2},M_{
t_1,t_2,f}>X^{\alpha}}\left|h\left(t_1^2-4,t_2^2-4,f\right)S_0\Phi\left
(y_1\left(S_0,T_0,F\right)\right)\right|\label{104}\end{equation}
is $\ll\frac {d^{5/2}}{\sqrt {X}}X^{-\beta}$.
\end{lemma}

\begin{proof}By Lemmas \ref{lem:3} and \ref{lem:4} we
can see that it is enough to consider the parts where
$S\left(t_1^2-4,t_2^2-4,f^2\right)\le X^{\alpha}$, and in the case of $
Z_{i,j_1,j_2}$ it is
enough to consider that part where we have also
$\left|t_1-t_2\right|\ge X^{\frac 12-\alpha}$. If $\alpha$ is small enough, then by Lemmas
\ref{lem:5}, \ref{lem:6}, \ref{lem:7,5} and by \cite[Lemma 3.1]{Biro} we obtain the
present lemma.
\end{proof}

\subsection{Using the explicit formula for the class
numbers}

\begin{lemma}\label{lem:9}Let $t_1,t_2$ and $f$ be positive integers such that $
t_1,t_2>2$
and $f^2>\left(t_1^2-4\right)\left(t_2^2-4\right)$. Let $p$ be a given prime. If
\[\beta_p:=\nu_p\left(f^2-\left(t_1^2-4\right)\left(t_2^2-4\right
)\right)\]
 and the residues of $t_1,t_2$ and $f$ are given modulo
$p^{\beta_p+1+\nu_p\left(16\right)}$, then the Hilbert symbol
$\left(t_2^2-4,f^2-\left(t_1^2-4\right)\left(t_2^2-4\right)\right
)_p$ is also determined.
\end{lemma}
\begin{proof}Let $\alpha_p:=\nu_p\left(t_2^2-4\right)$ and $t_2^2
-4=p^{\alpha_p}u$,
$f^2-\left(t_1^2-4\right)\left(t_2^2-4\right)=p^{\beta_p}v$. Then by \cite[Theorem 1 on
p 20]{Serre} we have that
\[\left(t_2^2-4,f^2-\left(t_1^2-4\right)\left(t_2^2-4\right)\right
)_p=\left(-1\right)^{\alpha_p\beta_p\frac {p-1}2}\left(\frac up\right
)^{\beta_p}\left(\frac vp\right)^{\alpha_p}\]
for $p>2$, and
\[\left(t_2^2-4,f^2-\left(t_1^2-4\right)\left(t_2^2-4\right)\right
)_2=\left(-1\right)^{\frac {u-1}2\frac {v-1}2+\alpha_p\frac {v^2-
1}8+\beta_p\frac {u^2-1}8}.\]
Since the residue of $t_2$ is given modulo $p^{\beta_p+1+\nu_p\left
(4\right)}$, the
fact whether $\alpha_p\ge\beta_p+1+\nu_p\left(4\right)$ or not is also
determined. In addition, in the case $\alpha_p<\beta_p+1+\nu_p\left
(4\right)$
the value of $\alpha_p$ is also determined.

Assume first that $\alpha_p\ge\beta_p+1+\nu_p\left(4\right)$. Then
\[f^2\equiv p^{\beta_p}v\left({\rm m}{\rm o}{\rm d}\mbox{\rm \ }p^{
\beta_p+1+\nu_p\left(4\right)}\right).\]
This implies that $\beta_p$ is even, and if $p>2$, then $\left(\frac
vp\right)=1$,
while if $p=2$, then $v\equiv 1\left({\rm m}{\rm o}{\rm d}\mbox{\rm \ }
8\right).$ Hence in this case
$\left(t_2^2-4,f^2-\left(t_1^2-4\right)\left(t_2^2-4\right)\right
)_p=1$, i.e. it is determined.

Assume now that $\alpha_p<\beta_p+1+\nu_p\left(4\right)$. We have seen that
then $\alpha_p$ is determined. Since the residue of
$t_2$ is given modulo $p^{\beta_p+1+\nu_p\left(16\right)}$, it is given also modulo
$p^{\alpha_p+1+\nu_p\left(4\right)}$. Hence $u$ is given modulo $
p^{1+\nu_p\left(4\right)}$. We
know that the residue of $f^2-\left(t_1^2-4\right)\left(t_2^2-4\right
)$ is given modulo
$p^{\beta_p+1+\nu_p\left(4\right)}$, therefore $v$ is also given modulo $
p^{1+\nu_p\left(4\right)}$.
The lemma follows.
\end{proof}

\begin{lemma}\label{lem:10}Let $t_1,t_2$ and $f$ be
integers such that $f^2\neq\left(t_1^2-4\right)\left(t_2^2-4\right
)$ and
$4|f^2-\left(t_1^2-4\right)\left(t_2^2-4\right)$. Let $E_{t_1,f}$ and $
G_{t_1,t_2,f}$ be as in
Lemma \ref{lem:7,5}.

If $p$ is a prime such that $\left(p,G_{t_1,t_2,f}\right)=1$, but
$p|f^2-\left(t_1^2-4\right)\left(t_2^2-4\right)$, then
\begin{equation}h_p\left(t_1^2-4,t_2^2-4,f\right)=\sum_{\beta =0}^{
\nu_p\left(f^2-\left(t_1^2-4\right)\left(t_2^2-4\right)\right)}\left
(\frac {t_1^2-4}{p^{\beta}}\right).\label{200}\end{equation}
Let $p$ be a given prime such that $p|G_{t_1,t_2,f}$. If $\beta_p
:=\nu_p\left(f^2-\left(t_1^2-4\right)\left(t_2^2-4\right)\right)$ and the residues of $
t_1,t_2$ and $f$ are given modulo
$p^{\beta_p+1+\nu_p\left(16\right)}$, then
\[h_p\left(t_1^2-4,t_2^2-4,f\right)\]
 is also determined.
\end{lemma}
\begin{proof}
Assume first that $\left(p,G_{t_1,t_2,f}\right)=1$ and
$p|f^2-\left(t_1^2-4\right)\left(t_2^2-4\right)$. We cannot have $
p|t_1^2-4$, because
this would imply $p|f$, hence $p|E_{t_1,f}$, which is a
contradiction. It is also clear that $p\ne 2$. Then we get
\eqref{200} by the definition of $h_p\left(t_1^2-4,t_2^2-4,f\right
)$
given on \cite[pp 898-899]{biroClass}.

Assume now that $p|G_{t_1,t_2,f}$. It is clear that
\begin{equation}m_p:=\min\left(\nu_p\left(t_1^2-4\right),\nu_p\left
(t_2^2-4\right),f\right)\le\frac {\beta_p}2.\label{22}\end{equation}
Hence the residues of $t_1,t_2$ and $f$ modulo $p^{\beta_p+1+\nu_
p\left(16\right)}$
determine the value of $m_p$.

Consider first the case $p\ne 2$. Then using Lemma
\ref{lem:9} and the definition of $h_p\left(t_1^2-4,t_2^2-4,f\right
)$
given on \cite[pp 898-899]{biroClass} we
see that in order to show that $h_p\left(t_1^2-4,t_2^2-4,f\right)$ is
determined it is enough to see that $\delta_p$ (defined there) is
determined modulo $p$. For this sake it is enough to
show that $t_1^2-4$ and $t_2^2-4$ are determined modulo $p^{m_p+1}
.$
But this is true, since $t_1$ and $t_2$ are determined modulo
$p^{\beta_p+1}$ and $m_p\le\beta_p$.

If $p=2$, then using Lemma \ref{lem:9} and the definition of $h_2\left
(t_1^2-4,t_2^2-4,f\right)$
given on \cite[p 899]{biroClass} we see
that it is enough to show that $A$ and $\epsilon$ (defined there)
are determined. This can be proved by a straightforward
but tedious reasoning, considering many cases separately.
We do not give the details, we just mention the
following two useful facts. On the one hand, it is clear by the definitions that
\[A\le\min\left(\nu_2\left(t_1^2-4\right),\nu_2\left(t_2^2-4\right
),f\right)\]
in every case, hence $A\le\frac {\beta_2}2$ by \eqref{22}. On the other
hand, in the case of \cite[(1.3)]{biroClass} we
must have $2\nu_2\left(t_1^2-4\right)\le\beta_2$, and in the case of
\cite[(1.4)]{biroClass} we must have
\[2\min\left(\nu_2\left(t_1^2-4\right),\nu_2\left(t_2^2-4\right)\right
)+2\le\beta_2.\]
The lemma is proved.
\end{proof}

Recall $A_X\left(t_1,t_2\right)$ and $B_X\left(t_1,t_2\right)$ from \eqref{79,1} and
\eqref{79,2}, respectively.

\begin{lemma}\label{lem:11}Let $\alpha >0$ and let $\delta >0$ be small enough in terms of $
\alpha$. Then there is a $\beta >0$ such that the following
statement is true.

If \eqref{201} holds and the positive integer $J$ is large enough in terms of $\alpha$
and $\delta$, then we have for every $1\le\tau\le 2$ that
\begin{equation}I_{d,X,J}\left(\tau\right)\ll_{\alpha ,\delta ,J}\frac {
d^{5/2}}{\sqrt {X}}X^{-\beta}+\Sigma_1+\Sigma_2,\label{202}\end{equation}
where $\Sigma_1$ and $\Sigma_2$ are defined as follows.

We define $\Sigma_1$ as
\[\sum_{1\le G,R\le X^{\alpha},4|G,p|2R\Leftrightarrow p|G}\sum_{\left
(A,B,C\right)\in H_{G,R}}\left(\prod_{p|G}h_p\left(A^2-4,B^2-4,C\right
)\right)S_{G,R,A,B,C}\]
with the notation
\[S_{G,R,A,B,C}:=\sum_{i\in\mathcal I}\sum_{j_1=0}^J\sum_{j_2=0}^
Ja_{j_1,j_2}X_{G,R,A,B,C}\left(i,j_1,j_2\right),\]
where $X_{G,R,A,B,C}\left(i,j_1,j_2\right)$ denotes
\[\sum_{\left(t_1,t_2,f\right)\in U_{i,j_1,j_2}\left(G,R,A,B,C\right
)}\alpha_G\left(t_1,t_2,f\right)A_X\left(t_1,t_2\right)\left(B\left
(S_0,T_0\right)-F\right)^{3/2},\]
and
\[\alpha_G\left(t_1,t_2,f\right):=\sum_{D|f^2-\left(t_1^2-4\right
)\left(t_2^2-4\right),\gcd\left(D,G\right)=1}\left(\frac {t_1^2-4}
D\right).\]
We define $\Sigma_2$ as
\[\sum_{1\le G,R\le X^{\alpha},4|G,p|2R\Leftrightarrow p|G}\sum_{\left
(A,B,C\right)\in H_{G,R}}\left(\prod_{p|G}h_p\left(A^2-4,B^2-4,C\right
)\right)T_{G,R,A,B,C}\]
with the notation
\[T_{G,R,A,B,C}:=\sum_{i\in\mathcal I}\sum_{j_1=0}^J\sum_{j_2=0}^
Ja_{j_1,j_2}Y_{G,R,A,B,C}\left(i,j_1,j_2\right),\]
where $Y_{G,R,A,B,C}\left(i,j_1,j_2\right)$ denotes
\[\sum_{\left(t_1,t_2,f\right)\in Z_{i,j_1,j_2}\left(G,R,A,B,C\right
)}\alpha_G\left(t_1,t_2,f\right)C_X\left(t_1,t_2\right)\left(A\left
(S_0,T_0\right)-F\right)^{3/2}\]
and
\begin{equation}C_X\left(t_1,t_2\right):=\operatorname{sgn}(t_1-t_
2)B_X\left(t_1,t_2\right)\psi_0\left(\frac {\left|t_1-t_2\right|}{
X^{\frac 12-\alpha}}\right).\label{678}\end{equation}
Here we use the following notations:
\[\mathcal I:=\left\{1\le i\le I:\left(a_i,\sqrt {X}-a_i\right)\ge
X^{\frac 12-\alpha}\right\},\]
\begin{equation}\gamma\left(G,R\right):=16RG\prod_{p|G}p,\label{679}\end{equation}
the set $H_{G,R}$ consists of the integer triples $\left(A,B,C\right
)$
satisfying that $0\le A,B,C<\gamma\left(G,R\right)$,
\begin{equation}4|C^2-\left(A^2-4\right)\left(B^2-4\right)\label{91}\end{equation}
and
\begin{equation}\prod_{p|G}p^{\nu_p\left(C^2-\left(A^2-4\right)\left
(B^2-4\right)\right)}=G,\;\prod_{p|G}p^{\nu_p\left(A^2-4\right)}=
R,\label{680}\end{equation}
the set $U_{i,j_1,j_2}\left(G,R,A,B,C\right)$ consists of the integer triples $\left
(t_1,t_2,f\right)$ satisfying the following
conditions:
\begin{equation}B\left(S_0\left(j_1,t_1\right),T_0\left(0,t_2\right
)\right)-\frac d{a_i^2}X^{\delta}<F\le B\left(S_0\left(j_1,t_1\right
),T_0\left(j_2,t_2\right)\right),\label{7,90}\end{equation}
\begin{equation}F>1,\label{80}\end{equation}
\begin{equation}a_i\le t_1,t_2\le a_{i+1}-1,\label{900}\end{equation}
\begin{equation}t_1\equiv A\left({\rm m}{\rm o}{\rm d}\,\gamma\left
(G,R\right)\right),t_2\equiv B\left({\rm m}{\rm o}{\rm d}\,\gamma\left
(G,R\right)\right),f\equiv C\left({\rm m}{\rm o}{\rm d}\,\gamma\left
(G,R\right)\right),\label{95}\end{equation}
\begin{equation}p|2\gcd\left(t_1^2-4,f\right)\Leftrightarrow p|G,\label{96}\end{equation}
finally $Z_{i,j_1,j_2}\left(G,R,A,B,C\right)$ is the set of integer triples
$\left(t_1,t_2,f\right)$ satisfying \eqref{80}, \eqref{900}, \eqref{95},
\eqref{96} and the following:
\begin{equation}A\left(S_0\left(j_1,t_1\right),T_0\left(0,t_2\right
)\right)-\frac {d\left|t_2-t_1\right|}{a_i\left(X-a_i^2\right)}X^{
\delta}<F\le A\left(S_0\left(j_1,t_1\right),T_0\left(j_2,t_2\right
)\right).\label{110}\end{equation}
\end{lemma}
\begin{proof}
This is a more or less straightforward consequence of
Lemmas \ref{lem:2}, \ref{lem:3}, \ref{lem:4}, \ref{lem:5},
\ref{lem:6}, \ref{lem:8}, \ref{lem:10} and \cite[Theorem 1.1]{biroClass}. We just make a few
remarks.

Under the condition \eqref{201} the term \eqref{7} can be
estimated by $\frac {d^{5/2}}{\sqrt {X}}X^{-\beta}$.

If the positive integers $G$ and $R$ are given such that $4|G$
and $p|2R\Leftrightarrow p|G$, then for a triple $\left(t_1,t_2,f\right
)$ satisfying
$4|f^2-\left(t_1^2-4\right)\left(t_2^2-4\right)>0$ we have that $
G_{t_1,t_2,f}=G$ and
$R_{t_1,f}=R$ hold if and only if \eqref{96}  is true and
\eqref{95} holds with some $\left(A,B,C\right)\in H_{G,R}$.

The lemma is proved.
\end{proof}

\begin{lemma}\label{lem:11,9}Assume \eqref{201}. Let us use the notations of Lemma \ref{lem:11}.
Assume that there is a $\beta_0>0$ and a small enough $\beta_1>0$ such that for every small enough
$\alpha >0$, for every $\delta >0$ which is small enough in terms of
$\alpha$, for every positive integer $J$ which is large enough
in terms of $\alpha$ and $\delta$, for every integers $R$ and $G$ satisfying that
$1\le G,R\le X^{\alpha},$ $4|G$ and $p|2R\Leftrightarrow p|G$, for every
$\left(A,B,C\right)\in H_{G,R}$ and for every $i\in\mathcal I$ we have that
\begin{equation}\sum_{j_1=0}^J\sum_{j_2=0}^Ja_{j_1,j_2}X_{G,R,A,B
,C}^{\ast}\left(i,j_1,j_2\right)\ll_{\alpha ,J}\frac {d^{5/2}}{\sqrt {
X}}X^{-\beta_0}\label{601}\end{equation}
and
\begin{equation}\sum_{j_1=0}^J\sum_{j_2=0}^Ja_{j_1,j_2}Y_{G,R,A,B
,C}^{\ast}\left(i,j_1,j_2\right)\ll_{\alpha ,J}\frac {d^{5/2}}{\sqrt {
X}}X^{-\beta_0},\label{602}\end{equation}
where $X_{G,R,A,B,C}^{\ast}\left(i,j_1,j_2\right)$ denotes
\[\sum_{\left(t_1,t_2,f\right)\in U^{\ast}\left(G,R,A,B,C\right)}
\alpha_G\left(t_1,t_2,f\right)A_{X,i,j_1,j_2}^{\ast}\left(t_1,t_2
,f\right)\left(B\left(S_0,T_0\right)-F\right)^{3/2}\]
and $Y_{G,R,A,B,C}^{\ast}\left(i,j_1,j_2\right)$ denotes
\[\sum_{\left(t_1,t_2,f\right)\in U^{\ast}\left(G,R,A,B,C\right)}
\alpha_G\left(t_1,t_2,f\right)C_{X,i.j_1,j_2}^{\ast}\left(t_1,t_2
,f\right)\left(A\left(S_0,T_0\right)-F\right)^{3/2},\]
$U^{\ast}\left(G,R,A,B,C\right)$ is the set of integer triples $\left
(t_1,t_2,f\right)$
satisfying \eqref{95} and \eqref{96}, and
\begin{equation}A_{X,i,j_1,j_2}^{\ast}\left(t_1,t_2,f\right):=A_X\left
(t_1,t_2\right)B_{X,i,j_1,j_2}\left(t_1,t_2,f\right),\label{602,1}\end{equation}
\begin{equation}C_{X,i.j_1,j_2}^{\ast}\left(t_1,t_2,f\right):=C_X\left
(t_1,t_2\right)D_{X,i,j_1,j_2}\left(t_1,t_2,f\right),\label{602,2}\end{equation}
where $B_{X,i,j_1,j_2}\left(t_1,t_2,f\right)$ denotes the product of
\begin{equation}\prod_{k=1}^2f_{\beta_1}\left(\frac {t_k-a_i}{a_{
i+1}-a_i}\right),\label{602,3}\end{equation}
\begin{equation}\psi_0\left(X^{\beta_1}\frac {\sqrt {t_1^2-4}\sqrt {
t_2^2-4}B\left(S_0\left(j_1,t_1\right),T_0\left(j_2,t_2\right)\right
)-\left|f\right|}d\right)\label{602,31}\end{equation}
and
\begin{equation}1-\psi_0\left(\frac {\sqrt {t_1^2-4}\sqrt {t_2^2-
4}\left(B\left(S_0\left(j_1,t_1\right),T_0\left(0,t_2\right)\right
)-\frac d{a_i^2}X^{\delta}\right)-\left|f\right|}d\right),\label{602,32}\end{equation}
while $D_{X,i,j_1,j_2}\left(t_1,t_2,f\right)$ denotes the product of \eqref{602,3},
\begin{equation}\psi_0\left(X^{\beta_1}\frac {\sqrt {t_1^2-4}\sqrt {
t_2^2-4}A\left(S_0\left(j_1,t_1\right),T_0\left(j_2,t_2\right)\right
)-\left|f\right|}d\right)\label{602,33}\end{equation}
and
\begin{equation}1-\psi_0\left(\frac {\sqrt {t_1^2-4}\sqrt {t_2^2-
4}\left(A\left(S_0\left(j_1,t_1\right),T_0\left(0,t_2\right)\right
)-\frac {d\left|t_2-t_1\right|}{a_i\left(X-a_i^2\right)}X^{\delta}\right
)-\left|f\right|}d\right)\label{602,34}\end{equation}
Then the left-hand side of \eqref{202} is
\begin{equation}\ll\frac {d^{5/2}}{\sqrt {X}}X^{-\beta}\label{603}\end{equation}
with a positive $\beta .$
\end{lemma}
\begin{proof}Assume first that there is a positive number
$\beta_0>0$ such that for every small enough $\alpha >0$, for every $
\delta >0$ which is small enough in terms of
$\alpha$ and for every positive integer $J$ which is large enough
in terms of $\alpha$ and $\delta$, we have for every integers $R$ and $
G$ satisfying that
$1\le G,R\le X^{\alpha},$ $4|G$ and $p|2R\Leftrightarrow p|G$, for every
$\left(A,B,C\right)\in H_{G,R}$ and for every $i\in\mathcal I$ that
\begin{equation}\sum_{j_1=0}^J\sum_{j_2=0}^Ja_{j_1,j_2}X_{G,R,A,B
,C}\left(i,j_1,j_2\right)\ll_{\alpha ,\delta ,J}\frac {d^{5/2}}{\sqrt {
X}}X^{-\beta_0}\label{604}\end{equation}
and
\begin{equation}\sum_{j_1=0}^J\sum_{j_2=0}^Ja_{j_1,j_2}Y_{G,R,A,B
,C}\left(i,j_1,j_2\right)\ll_{\alpha ,\delta ,J}\frac {d^{5/2}}{\sqrt {
X}}X^{-\beta_0}.\label{605}\end{equation}
Then \eqref{603} with some positive $\beta$ follows easily from Lemma \ref{lem:11},
fixing $\alpha$ to be small enough there. Indeed, the number of
$G,R$, $\left(A,B,C\right)\in H_{G,R}$ and $i\in\mathcal I$ is smaller than
any fixed power of $X$ if $\alpha$ is small enough. In addition,
it is easy to show by trivial estimates using the
definitions on \cite[pp 898-899]{biroClass} that if $d_1$, $d_2$ and $t$ are integers such that $t^2-d_1d_2\neq
0$,
then for any prime $p$ we have that
$h_p$$\left(d_1,d_2,t\right)\le p^{C\nu_p\left(t^2-d_1d_2\right)}$ with an absolute constant $
C$.
Using this estimate and the first relation of \eqref{680}
we obtain indeed that \eqref{603} with some positive $\beta$
follows if we assume \eqref{604} and \eqref{605} under
the conditions listed above.

Hence it is enough to show that for every $\beta_1>0$ there is
a $\beta_2>0$ such that we have for every small enough $\alpha >0$, for every $
\delta >0$ which is small enough in terms of
$\alpha$ and for every positive integer $J$ which is large enough
in terms of $\alpha$ and $\delta$ that for every $R,G,A,B,C$ and $
i$ with
the properties listed in the lemma we have
\[\sum_{j_1=0}^J\sum_{j_2=0}^Ja_{j_1,j_2}\left(X_{G,R,A,B,C}\left
(i,j_1,j_2\right)-X_{G,R,A,B,C}^{\ast}\left(i,j_1,j_2\right)\right
)\ll_{\alpha ,\delta ,J}\frac {d^{5/2}}{\sqrt {X}X^{\beta_2}}\]
and
\[\sum_{j_1=0}^J\sum_{j_2=0}^Ja_{j_1,j_2}\left(Y_{G,R,A,B,C}\left
(i,j_1,j_2\right)-Y_{G,R,A,B,C}^{\ast}\left(i,j_1,j_2\right)\right
)\ll_{\alpha ,\delta ,J}\frac {d^{5/2}}{\sqrt {X}X^{\beta_2}}.\]
To see this we note on the one hand that inserting the
factor \eqref{602,3} instead of the condition \eqref{900} we make an admissible error. Indeed, this
follows easily from the fact that if $\alpha$ and $\delta$ are small enough, then the trivial upper bound for
the left-hand sides of \eqref{601}, \eqref{602}, \eqref{604}
and \eqref{605} is $\frac {d^{5/2}X^{\epsilon}}{\sqrt {X}}$ for any $
\epsilon >0$, and taking the
differences the trivial upper bound becomes smaller. We
use here that because of the presence of the
factor $\psi_0\left(\frac {\left|t_1-t_2\right|}{X^{\frac 12-\alpha}}\right
)$ in \eqref{678} we have \eqref{741} if $\alpha$ and $\delta$ are small enough.

Completely similarly, inserting the factor \eqref{602,31} instead of the right inequality of
the condition \eqref{7,90}, and inserting the factor \eqref{602,33} instead of the right inequality of
the condition \eqref{110} we make admissible errors.

On the other hand, inserting the factor \eqref{602,32} instead of the left inequality of
the condition \eqref{7,90}, and inserting the factor
\eqref{602,34} instead of the left inequality of the
condition \eqref{110} we make admissible errors. To
see this we note that there is a $c>0$ such that if $\beta_1,\alpha
>0$ are small enough,
\[a_i,\sqrt {X}-a_i,\gg X^{\frac 12-\alpha},\]
$\delta$ is small enough in terms of $\beta_1,\alpha$, and $J$ is large enough in terms
of $\delta$, then for any $0\le j_1\le J$ the following two statements are true:

For any $t_1,t_2$ such that \eqref{602,3} is nonzero and for any $
f$ satisfying that
\begin{equation}F\le B\left(S_0\left(j_1,t_1\right),T_0\left(0,t_
2\right)\right)-\frac d{a_i^2}X^{\delta}\label{700,5}\end{equation}
and that \eqref{602,32} is nonzero we have that
\begin{equation}\sum_{j_2=0}^Ja_{j_1,j_2}\left(B\left(S_0\left(j_
1,t_1\right),T_0\left(j_2,t_2\right)\right)-F\right)^{3/2}\ll\left
(\frac dX\right)^{3/2}X^{-c}.\label{701}\end{equation}
For any $t_1,t_2$ such that \eqref{602,3} is nonzero and
$\left|t_2-t_1\right|\gg X^{\frac 12-\alpha}$and for any $f$ satisfying that
\begin{equation}F\le A\left(S_0\left(j_1,t_1\right),T_0\left(0,t_
2\right)\right)-\frac {d\left|t_2-t_1\right|}{a_i\left(X-a_i^2\right
)}X^{\delta}\label{701,5}\end{equation}
and that \eqref{602,34} is nonzero we have that
\begin{equation}\sum_{j_2=0}^Ja_{j_1,j_2}\left(A\left(S_0\left(j_
1,t_1\right),T_0\left(j_2,t_2\right)\right)-F\right)^{3/2}\ll\left
(\frac dX\right)^{3/2}X^{-c}.\label{702}\end{equation}
Indeed, \eqref{701} follows from Lemma \ref{lem:5} and
\cite[(6.62)]{Biro}, while \eqref{702} follows from Lemma
\ref{lem:6} and \cite[(6.51)]{Biro}.

We also use that \eqref{700,5} implies the right
inequality of \eqref{7,90} and \eqref{701,5} implies the right
inequality of \eqref{110}. This follows from \cite[(6.28)]{Biro} and from \cite[(6.29)]{Biro},
respectively.

The lemma follows.
\end{proof}

\subsection{A new expression for $\alpha_G\left(t_1,t_2,f\right)$}

Let $\gamma_0$ be a smooth nonnegative even function on the real line
such that $\gamma_0\left(y\right)=0$ for $y\ge 1$, and $\gamma_0\left
(y\right)+\gamma_0\left(1-y\right)=1$
for $0\le y\le 1$. Such a function clearly exists, and it is easy to see that
\[\sum_{m=-\infty}^{\infty}\gamma_0\left(y-m\right)=1\]
for every real $y$.

Let $\phi_0\left(x\right):=\gamma_0\left(\log x\right)$ for $x>0$. Then $
\phi_0$
is a smooth nonnegative function on the positive real axis,
$\phi_0\left(x\right)=\phi_0\left(1/x\right)$ for every $x>0$, $\phi_
0\left(x\right)=0$ for $x>e$ and
for $x<1/e$ and we have
\[\sum_{m=-\infty}^{\infty}\phi_0\left(\frac x{e^m}\right)=1\]
for every $x>0$. Let $\Phi_0$ be the Mellin transform of $\phi_0$.

\begin{lemma}\label{lem:12,9} Let $G$ and $R$ be positive integers such that $
4|G$ and
$p|2R\Leftrightarrow p|G$, and let $\left(A,B,C\right)\in H_{G,R}$. Let $
t_1,t_2>2$ and $f$ be integers such that
$f^2>\left(t_1^2-4\right)\left(t_2^2-4\right)$, $f^2\le 4X^2$ and \eqref{95} and \eqref{96}
are true. Let us write $N:=\frac {f^2-\left(t_1^2-4\right)\left(t_
2^2-4\right)}G$. Then $N$ is
an integer, $\gcd\left(N,G\right)=1$ and we have that $\alpha_G\left
(t_1,t_2,f\right)$ equals the sum of
\[\sum_{1\le e^m\le 2eX}\sum_{D|N}\left(\frac {t_1^2-4}D\right)\phi_
0\left(\frac D{e^m}\right)\]
and
\[\kappa\sum_{2eX<e^m\le 4eX^2}\sum_{1\le e^r\le 2eX}\sum_{D|N}\left
(\frac {t_1^2-4}D\right)\phi_0\left(\frac D{e^r}\right)\frac 1{2\pi
i}\int_{\left(0\right)}\left(\frac {De^m}N\right)^{-s}\Phi_0\left
(s\right)ds,\]
where $m$ and $r$ run over integers, $\kappa :=\kappa_{G,R,A,B,C}$ depends only on $
G,R,A,B$ and $C$, and we have $\kappa_{G,R,A,B,C}\in \{-1,1\}$.
\end{lemma}
\begin{proof}
By \eqref{96}, \eqref{680}, \eqref{95} and \eqref{679} we have that $
p|2\gcd\left(t_1^2-4,f\right)\Leftrightarrow p|G$ and
\[\prod_{p|G}p^{\nu_p\left(f^2-\left(t_1^2-4\right)\left(t_2^2-4\right
)\right)}=G,\;\prod_{p|G}p^{\nu_p\left(t_1^2-4\right)}=R.\]
In particular, $N$ is an integer and $\gcd\left(N,G\right)=1$.

Then we have that if $D|f^2-\left(t_1^2-4\right)\left(t_2^2-4\right
)$ and $\gcd\left(D,G\right)=1$, then
\begin{equation}\left(\frac {t_1^2-4}D\right)=\left(\frac RD\right
)\left(\frac {\left(t_1^2-4\right)/R}D\right).\label{501}\end{equation}
Since $2|G$, we see that $\left(t_1^2-4\right)/R$ and $D$ are odd. By
quadratic reciprocity we have
\begin{equation}\left(\frac {\left(t_1^2-4\right)/R}D\right)=\left
(-1\right)^{\frac {D-1}2\frac {\frac {t_1^2-4}R-1}2}\left(\frac D{\left
(t_1^2-4\right)/R}\right),\label{502}\end{equation}
and
\begin{equation}\left(\frac D{\left(t_1^2-4\right)/R}\right)=\left
(\frac {\left(f^2-\left(t_1^2-4\right)\left(t_2^2-4\right)\right)
/D}{\left(t_1^2-4\right)/R}\right).\label{503}\end{equation}
To see this last statement we have to check that
$\gcd\left(f,\frac {t_1^2-4}R\right)=1$. But this is true, since a common prime
divisor would divide $G$, but $\gcd\left(G,\frac {t_1^2-4}R\right
)=1$.
We see that
\begin{equation}\left(\frac {\left(f^2-\left(t_1^2-4\right)\left(
t_2^2-4\right)\right)/D}{\left(t_1^2-4\right)/R}\right)=\left(\frac {
D^{\ast}}{\left(t_1^2-4\right)/R}\right)\left(\frac G{\left(t_1^2
-4\right)/R}\right),\label{504}\end{equation}
where
\begin{equation}D^{\ast}:=\frac {f^2-\left(t_1^2-4\right)\left(t_
2^2-4\right)}{DG}.\label{505}\end{equation}
Again by qauadratic reciprocity we get
\begin{equation}\left(\frac {D^{\ast}}{\left(t_1^2-4\right)/R}\right
)=\left(-1\right)^{\frac {D^{\ast}-1}2\frac {\frac {t_1^2-4}R-1}2}\left
(\frac {t_1^2-4}{D^{\ast}}\right)\left(\frac R{D^{\ast}}\right).\label{506}\end{equation}
By \eqref{501}, \eqref{502}, \eqref{503}, \eqref{504},
\eqref{506} and \eqref{505} we obtain
\begin{equation}\left(\frac {t_1^2-4}D\right)=\left(\frac {t_1^2-
4}{D^{\ast}}\right)\kappa\label{507}\end{equation}
with
\[\kappa :=\left(\frac R{\frac {f^2-\left(t_1^2-4\right)\left(t_2^
2-4\right)}G}\right)\left(-1\right)^{\frac {D+D^{\ast}-2}2\frac {\frac {
t_1^2-4}R-1}2}\left(\frac G{\left(t_1^2-4\right)/R}\right).\]
We have that $4|$$\left(D-1\right)\left(D^{\ast}-1\right)$, since $
D$ and $D^{\ast}$ are odd.
Therefore
\[\frac {D+D^{\ast}-2}2-\frac {DD^{\ast}-1}2\]
is even. Consequently, we get by \eqref{505} that $\kappa$ equals
\[\left(\frac R{\frac {f^2-\left(t_1^2-4\right)\left(t_2^2-4\right
)}G}\right)\left(-1\right)^{\frac {\frac {f^2-\left(t_1^2-4\right
)\left(t_2^2-4\right)}G-1}2\frac {\frac {t_1^2-4}R-1}2}\left(\frac
G{\left(t_1^2-4\right)/R}\right).\]
We see by \eqref{679} and \eqref{95} that $t_1$, $t_2$ and $f$
are fixed modulo $4GR$, hence $\kappa$ is determined by $G,R,A,B$
and $C$. Then by \eqref{505} and \eqref{507} we obtain
that $\alpha_G\left(t_1,t_2,f\right)$ equals the sum of
\[\sum_{1\le e^m\le 2eX}\sum_{D|N}\left(\frac {t_1^2-4}D\right)\phi_
0\left(\frac D{e^m}\right)\]
and
\[\kappa\sum_{2eX<e^m\le 4eX^2}\sum_{D|N}\left(\frac {t_1^2-4}D\right
)\phi_0\left(\frac {De^m}N\right).\]
Here we must have $\frac {De^m}N\le e$, but $N\le 4X^2,$ $e^m>2eX$,
hence $1\le D\le 2X$. So we may insert the factor
$\sum_{1\le e^r\le 2eX}\phi_0\left(\frac D{e^r}\right)$. The lemma follows by Mellin
inversion.
\end{proof}

\subsection{Applying the Poisson formula}

\begin{lemma}\label{lem:11,94}Assume \eqref{201}. Let us use
the notations of Lemma \ref{lem:11,9}. Assume that there are positive numbers $
\beta_0,\beta_1,\beta_2>0$ such
that $\beta_1$ is small enough, and for every small enough $\alpha
>0$, for every $\delta >0$ which is small enough in terms of
$\alpha$, for every positive integer $J$ which is large enough
in terms of $\alpha$ and $\delta$, for every integers $R$ and $G$ satisfying that
$1\le G,R\le X^{\alpha},$ $4|G$ and $p|2R\Leftrightarrow p|G$, for every
$\left(A,B,C\right)\in H_{G,R}$, for every $i\in\mathcal I$, for every
integer $1\le k\le X^{\beta_2}$ with $\gcd\left(k,G\right)=1$, for every purely
imaginary $s$ with $|s|\le X^{\beta_2}$, and for every integer $r$ with
$1\le e^r\le 2eX$ we have writing $\gamma :=\gamma\left(G,R\right
)$ that
\begin{equation}\sum_{\left(j_1,j_2,\,t_1,t_2,\,D,\,N\right)\in H_{
J,A,B,k,G}}a_{j_1,j_2}\frac {\left(\frac {t_1^2-4}D\right)\phi_0\left
(\frac D{e^r}\right)A_N^{(q)}}{D^{1+s}}\ll_{\alpha ,J}\frac {d^{5
/2}}{\sqrt {X}}X^{-\beta_0},\label{911}\end{equation}
for $q=1,2$, where $H_{J,A,B,k,G}$ is the set of
$\left(j_1,j_2,\,t_1,t_2,\,D,\,N\right)\in\mathbb{Z}^6$ satisfying the conditions
\begin{equation}0\le j_1,j_2\le J,\,t_1,t_2>2,\,D\ge 1,\label{912}\end{equation}
\begin{equation}\,t_1\equiv A\left({\rm m}{\rm o}{\rm d}\,\gamma\right
),t_2\equiv B\left({\rm m}{\rm o}{\rm d}\,\gamma\right),\,k|t_1^2
-4,\,\gcd\left(D,kG\right)=1,\label{913}\end{equation}
and
\[A_N^{(q)}:=B_N^{(q)}C_N,\]
where $C_N$ denotes
\[\sum_{h\:{\rm m}{\rm o}{\rm d}\:\gamma D,kh\equiv C\left({\rm m}
{\rm o}{\rm d}\,\gamma\right),D|k^2h^2-\left(t_1^2-4\right)\left(
t_2^2-4\right)}e\left(\frac {Nh}{\gamma D}\right),\]
finally
\[B_N^{(q)}:=\int_{-\infty}^{\infty}\left(\left(t_1^2-4\right)\left
(t_2^2-4\right)\right)^sG_q\left(h,t_2\right)e\left(-\frac N{\gamma
D}h\right)dh\]
with
\begin{equation}G_1\left(h,t_2\right):=A_{X,i,j_1,j_2}^{\ast}\left
(t_1,t_2,kh\right)\left(B\left(S_0,T_0\right)-F\right)^{3/2}\left
(B^2\left(S_0,T_0\right)-1\right)^s,\label{913,1}\end{equation}
\begin{equation}G_2\left(h,t_2\right):=C_{X,i,j_1,j_2}^{\ast}\left
(t_1,t_2,kh\right)\left(A\left(S_0,T_0\right)-F\right)^{3/2}\left
(A^2\left(S_0,T_0\right)-1\right)^s,\label{913,2}\end{equation}
and $F$ is given by \eqref{742} with  $f=kh$ there.

Then the left-hand side of \eqref{202} is
\[\ll\frac {d^{5/2}}{\sqrt {X}}X^{-\beta}\]
with a positive $\beta .$
\end{lemma}
\begin{proof}Let $G,R,A,B$ and $C$ be as in the lemma and
assume \eqref{95} for some $t_1,t_2,f$. If $p|G$, $p\neq 2$, then $
p|2R\Leftrightarrow p|G$ implies $p|R$,
hence \eqref{680} and \eqref{95} imply $\nu_p\left(t_1^2-4\right)
>0$ and
$\nu_p\left(f\right)>0$. Hence if \eqref{95} holds, then \eqref{96} is
true if and only if there is no such prime divisor $p$ of
$\gcd\left(t_1^2-4,f\right)$ for which $\gcd\left(p,G\right)=1$. Consequently, if an integer triple $\left
(t_1,t_2,f\right)$
satisfies \eqref{95}, then
\[\left(t_1,t_2,f\right)\in U^{\ast}\left(G,R,A,B,C\right)\Leftrightarrow
\sum_{k|\gcd\left(t_1^2-4,f\right),\gcd\left(k,G\right)=1}\mu\left
(k\right)=1,\]
\[\left(t_1,t_2,f\right)\notin U^{\ast}\left(G,R,A,B,C\right)\Leftrightarrow
\sum_{k|\gcd\left(t_1^2-4,f\right),\gcd\left(k,G\right)=1}\mu\left
(k\right)=0.\]
So in $X_{G,R,A,B,C}^{\ast}\left(i,j_1,j_2\right)$  and also in $
Y_{G,R,A,B,C}^{\ast}\left(i,j_1,j_2\right)$
one can insert the factor $\sum_{k|\gcd\left(t_1^2-4,f\right),\gcd\left
(k,G\right)=1}\mu\left(k\right)$
and extend the summation to integer triple $\left(t_1,t_2,f\right
)$
satisfying \eqref{95}. It is also clear that if $\beta_2>0$ is
given and $\alpha$ and $\beta_1$ are small enough in terms of $\beta_
2$, then
the contribution of $k>X^{\beta_2}$ in \eqref{601} and \eqref{602}
is $\ll\frac {d^{5/2}}{\sqrt {X}}X^{-c}$ with some $c>0$. For a given $
k\le X^{\beta_2}$ with
$\gcd\left(k,G\right)=1$ we write $f=kh$.

Next we substitute the expression given for $\alpha_G\left(t_1,t_
2,f\right)$
in Lemma \ref{lem:12,9}. Note that because of
the presence of the factor $\left(\frac {t_1^2-4}D\right)$ we may assume
$\gcd\left(D,k\right)=1$, since we have $k|t_1^2-4.\,$

We use that $\Phi_0\left(s\right)$ decays faster than polynomially, and we use also that for
$|s|\le X^{\beta_2}$ we have that $\left(f^2-\left(t_1^2-4\right)\left
(t_2^2-4\right)\right)^s$ equals
\[\left(\left(t_1^2-4\right)\left(t_2^2-4\right)\right)^s\left(B^
2\left(S_0,T_0\right)-1\right)^s\left(1+O\left(X^{-c}\right)\right
)\]
in the case $A_{X,i,j_1,j_2}^{\ast}\left(t_1,t_2,f\right)\neq 0$, and it equals
\[\left(\left(t_1^2-4\right)\left(t_2^2-4\right)\right)^s\left(A^
2\left(S_0,T_0\right)-1\right)^s\left(1+O\left(X^{-c}\right)\right
)\]
in the case $C_{X,i,j_1,j_2}^{\ast}\left(t_1,t_2,f\right)\neq 0$ with some positive $
c$ if $\beta_1$, $\beta_2$ and $\alpha$ are small enough.

Then we apply the Poisson summation formula in the
variable $h$. We have arithmetic weights depending on the
residue of $h$ modulo $\gamma\left(G,R\right)D$.

The lemma follows from these considerations.
\end{proof}

\subsection{The analytic and the discrete Fourier
transform}

\begin{lemma}\label{lem:11,95}Let us write
\[h_{\beta_1,C}\left(Y\right):=\int_0^{\infty}e\left(yY\right)y^{
3/2}\psi_0\left(X^{\beta_1}y\right)\left(1-\psi_0\left(y+C\right)\right
)dy,\]
and let us use the notations $B_N^{(1)}$ and $B_N^{(2)}$ from Lemma
\ref{lem:11,94}. Assume \eqref{201} and that $\delta$ is small enough and
\eqref{79} is true. Then we have that $B_N^{(1)}$ equals the product of
\[\frac {d^{5/2}A_X\left(t_1,t_2\right)\left(B^2\left(S_0,T_0\right
)-1\right)^s\prod_{k=1}^2\left(f_{\beta_1}\left(\frac {t_k-a_i}{a_{
i+1}-a_i}\right)\left(t_k^2-4\right)^s\right)}{k\left(\sqrt {t_1^
2-4}\sqrt {t_2^2-4}\right)^{3/2}}\]
and
\[\sum_{\sigma\in \{-1,1\}}h_{\beta_1,C_1\left(j_1,j_2\right)}\left
(\sigma\frac {Nd}{kD\gamma}\right)e\left(-\sigma\frac N{kD\gamma}\sqrt {
t_1^2-4}\sqrt {t_2^2-4}B\left(S_0,T_0\right)\right).\]
If we assume also that \eqref{79,15} is true, then $B_N^{(2)}$ equals the product of
\[\frac {d^{5/2}C_X\left(t_1,t_2\right)\left(A^2\left(S_0,T_0\right
)-1\right)^s\prod_{k=1}^2\left(f_{\beta_1}\left(\frac {t_k-a_i}{a_{
i+1}-a_i}\right)\left(t_k^2-4\right)^s\right)}{k\left(\sqrt {t_1^
2-4}\sqrt {t_2^2-4}\right)^{3/2}}\]
and
\[\sum_{\sigma\in \{-1,1\}}h_{\beta_1,C_2\left(j_1,j_2\right)}\left
(\sigma\frac {Nd}{kD\gamma}\right)e\left(-\sigma\frac N{kD\gamma}\sqrt {
t_1^2-4}\sqrt {t_2^2-4}A\left(S_0,T_0\right)\right),\]
where
\[C_1\left(j_1,j_2\right):=\frac {B\left(S_0\left(j_1,t_1\right),
T_0\left(0,t_2\right)\right)-\frac d{a_i^2}X^{\delta}-B\left(S_0\left
(j_1,t_1\right),T_0\left(j_2,t_2\right)\right)}{d\left(\left(t_1^
2-4\right)\left(t_2^2-4\right)\right)^{-1/2}},\]
\[C_2\left(j_1,j_2\right):=\frac {A\left(S_0\left(j_1,t_1\right),
T_0\left(0,t_2\right)\right)-\frac {d\left|t_2-t_1\right|X^{\delta}}{
a_i\left(X-a_i^2\right)}-A\left(S_0\left(j_1,t_1\right),T_0\left(
j_2,t_2\right)\right)}{d\left(\left(t_1^2-4\right)\left(t_2^2-4\right
)\right)^{-1/2}}.\]
\end{lemma}

\begin{proof}Using the definitions it is easy to compute that $B_
N^{(1)}$ equals
\[U_{N,C_1}\left(t_1,t_2\right)A_X\left(t_1,t_2\right)\left(B^2\left
(S_0,T_0\right)-1\right)^s\prod_{k=1}^2\left(f_{\beta_1}\left(\frac {
t_k-a_i}{a_{i+1}-a_i}\right)\left(t_k^2-4\right)^s\right)\]
and $B_N^{(2)}$ equals
\[V_{N,C_2}\left(t_1,t_2\right)C_X\left(t_1,t_2\right)\left(A^2\left
(S_0,T_0\right)-1\right)^s\prod_{k=1}^2\left(f_{\beta_1}\left(\frac {
t_k-a_i}{a_{i+1}-a_i}\right)\left(t_k^2-4\right)^s\right),\]
where $U_{N,C_1}\left(t_1,t_2\right)$ denotes
\[\int_{-\infty}^{\infty}e\left(-\frac N{\gamma D}h\right)\left(B\left
(S_0,T_0\right)-F\right)^{3/2}\psi_0\left(X^{\beta_1}y_1\right)\left
(1-\psi_0\left(y_1+C_1\right)\right)dh,\]
$V_{N,C_2}\left(t_1,t_2\right)$ denotes
\[\int_{-\infty}^{\infty}e\left(-\frac {Nh}{\gamma D}\right)\left
(A\left(S_0,T_0\right)-F\right)^{\frac 32}\psi_0\left(X^{\beta_1}
y_2\right)\left(1-\psi_0\left(y_2+C_2\right)\right)dh,\]
$F$ is given by \eqref{742} with $f=kh$ there, and we use
the abbreviations $C_1=C_1\left(j_1,j_2\right)$, $C_2=C_2\left(j_
1,j_2\right)$,
\begin{equation}y_1:=\frac {\sqrt {t_1^2-4}\sqrt {t_2^2-4}B\left(
S_0,T_0\right)-\left|kh\right|}d,\label{901}\end{equation}
\begin{equation}y_2:=\frac {\sqrt {t_1^2-4}\sqrt {t_2^2-4}A\left(
S_0,T_0\right)-\left|kh\right|}d.\label{902}\end{equation}
We consider separately the parts $h>0$ and $h<0$ and
make the substitutions \eqref{901} and \eqref{902}. We
first obtain in this way that
\begin{equation}-\infty <y_1\le\frac {\sqrt {t_1^2-4}\sqrt {t_2^2
-4}B\left(S_0,T_0\right)}d,\label{191}\end{equation}
\begin{equation}-\infty <y_2\le\frac {\sqrt {t_1^2-4}\sqrt {t_2^2
-4}A\left(S_0,T_0\right)}d.\label{192}\end{equation}
However, because of the factors $\psi_0\left(X^{\beta_1}y_i\right
)$ we can
assume $y_i\ge 0$. On the other hand, we can extend the
integrations to $\infty$, since if \eqref{191} does not hold, then
$1-\psi_0\left(y_1+C_1\right)=0$, and if \eqref{192} does not hold, then
$1-\psi_0\left(y_2+C_2\right)=0$. We obtain the lemma by some computations.
\end{proof}

\begin{lemma}\label{lem:14,9}Let $t_1,t_2>2$ and $D,\gamma ,k$ be positive
integers such that $\gcd\left(\gamma ,k\right)=1$ and
\[\gcd\left(D,2\left(t_1^2-4\right)\gamma k\right)=1.\]
Then for any integers $N,C$ we have that
\begin{equation}\sum_{h\:{\rm m}{\rm o}{\rm d}\:\gamma D,kh\equiv
C\left({\rm m}{\rm o}{\rm d}\,\gamma\right),D|k^2h^2-\left(t_1^2-
4\right)\left(t_2^2-4\right)}e\left(\frac {Nh}{\gamma D}\right)\label{3000}\end{equation}
equals
\[e\left(\frac {NC\overline {kD}}{\gamma}\right)\sum_{\lambda |D,
N}\frac {\overline {\epsilon_{D/\lambda}}}{\sqrt {\frac D{\lambda}}}\left
(\frac {t_1^2-4}{D/\lambda}\right)T\left(\left(t_1^2-4\right)\left
(\frac N{\lambda}\right)^2,\overline {4\gamma^2k^2}\left(t_2^2-4\right
);\frac D{\lambda}\right).\]
\end{lemma}
\begin{proof}The condition $D|k^2h^2-\left(t_1^2-4\right)\left(t_
2^2-4\right)$ can be
omitted by inserting the factor
\[\frac 1D\sum_{a\:{\rm m}{\rm o}{\rm d}\:D}e\left(\frac {a\left(
k^2h^2-\left(t_1^2-4\right)\left(t_2^2-4\right)\right)}D\right),\]
hence \eqref{3000} equals
\[\frac 1D\sum_{a\:{\rm m}{\rm o}{\rm d}\:D}e\left(-\frac {a\left
(t_1^2-4\right)\left(t_2^2-4\right)}D\right)T_a\]
with
\[T_a:=\sum_{h\:{\rm m}{\rm o}{\rm d}\:\gamma D,kh\equiv C\left({\rm m}
{\rm o}{\rm d}\,\gamma\right)}e\left(\frac {Nh}{\gamma D}\right)e\left
(\frac {ak^2h^2}D\right).\]
Let us  write $h=\alpha D+H\gamma$, where $H$ runs modulo
$D$, and $\alpha$ is such that $k\alpha D\equiv C\left({\rm m}{\rm o}
{\rm d}\,\gamma\right)$. Then
\[T_a=e\left(\frac {N\alpha}{\gamma}\right)\sum_{H\:{\rm m}{\rm o}
{\rm d}\:D}e\left(\frac {NH}D\right)e\left(\frac {ak^2H^2\gamma^2}
D\right).\]
Let us write $\gcd\left(a,D\right)=\lambda$, and let $D_1:=D/\lambda$ and
$a_1:=a/\lambda$. It is clear that if $\lambda$ does not divide $
N$, then $T_a=0$. If
$\lambda |N$, then
\[T_a=\lambda e\left(\frac {N\alpha}{\gamma}\right)\sum_{H\:{\rm m}
{\rm o}{\rm d}\:D_1}e\left(\frac {\left(N/\lambda\right)H}{D_1}\right
)e\left(\frac {a_1k^2H^2\gamma^2}{D_1}\right).\]
In general, if $C>0$ is an odd integer, $A,B$ are integers
such that $B$ and $C$ are relatively prime, then
\begin{equation}\sum_{x\:{\rm m}{\rm o}{\rm d}\:C}e\left(\frac {A
x+Bx^2}C\right)=e\left(-\frac {A^2\overline {4B}}C\right)\left(\frac
BC\right)\epsilon_C\sqrt {C}.\label{400}\end{equation}
Indeed, it is enough to show it for $A=0$, because the
general case follows by substituting $x$ in place of
$x+A\overline {2B}$. The $A=0$ case is proved in \cite[Lemma 4.8]{I}. Hence we have
shown \eqref{400}, and then we get for the case $\lambda |N$ that
\[T_a=\lambda e\left(\frac {N\alpha}{\gamma}\right)e\left(-\frac {\left
(N/\lambda\right)^2\overline {4a_1k^2\gamma^2}}{D_1}\right)\left(\frac {
a_1k^2\gamma^2}{D_1}\right)\epsilon_{D_1}\sqrt {D_1}.\]
Noting that $\alpha\equiv C\overline {kD}\left({\rm m}{\rm o}{\rm d}\,
\gamma\right)$, writing
\[x=-4\left(t_1^2-4\right)a_1k^2\gamma^2\]
and using $\left(\frac {-1}{D_1}\right)\epsilon_{D_1}=\overline {
\epsilon_{D_1}}$ we obtain the lemma.
\end{proof}

\subsection{Handlig a coprimality condition}

\begin{lemma}\label{lem:11,941}Assume \eqref{201}. Let us
use the notations of Lemmas \ref{lem:11,9} and
\ref{lem:11,94}. Then for every
$\xi >0$ there is a $\beta_0>0$ such that if $\beta_1$, $\beta_2$, $
\alpha$ and $\delta$ are small
enough in terms of $\xi$, then for every positive integer $J$ which is large enough
in terms of $\alpha$ and $\delta$, for every integers $R$ and $G$ satisfying that
$1\le G,R\le x^{\alpha},$ $4|G$ and $p|2R\Leftrightarrow p|G$, for every
$\left(A,B,C\right)\in H_{G,R}$, for every $i\in\mathcal I$, for every
integer $1\le k\le X^{\beta_2}$ with $\gcd\left(k,G\right)=1$, for every purely
imaginary $s$ with $|s|\le X^{\beta_2}$, and for every integer $r$ with
$1\le e^r\le 2eX$ we have writing $\gamma :=\gamma\left(G,R\right
)$ that \eqref{911}
equals
\begin{equation}\sum_{j_1,j_2,\,t_1,t_2,D}a_{j_1,j_2}\frac {\phi_
0\left(\frac D{e^r}\right)}{D^{1+s}}\left(\sum_{K|t_1^2-4,D,K\le
X^{\xi},\gcd\left(K,kG\right)=1}\mu\left(K\right)\right)\sum_{\lambda
|D}\frac {\left(\frac {t_1^2-4}{\lambda}\right)S_{\lambda}^{\ast}}{
\epsilon_{D/\lambda}\sqrt {\frac D{\lambda}}}\label{971}\end{equation}
plus $O\left(\frac {d^{5/2}}{\sqrt {X}}X^{-\beta_0}\right)$, where we sum in \eqref{971} under the
conditions \eqref{912}, \eqref{913}, and $S_{\lambda}^{\ast}$ abbreviates
\[\sum_{0\neq N\in\mathbb{Z},N\equiv 0\left({\rm m}{\rm o}{\rm d}\,
\lambda\right)}B_N^{(q)}e\left(\frac {NC\overline {kD}}{\gamma}\right
)T\left(\left(t_1^2-4\right)\left(\frac N{\lambda}\right)^2,\overline {
4\gamma^2k^2}\left(t_2^2-4\right);\frac D{\lambda}\right).\]
\end{lemma}

\begin{proof}By Lemma \ref{lem:14,9} we have that the left-hand side
of \eqref{911} equals
\begin{equation}\sum_{\left(j_1,j_2,\,t_1,t_2,\,D,\,N\right)\in H_{
J,A,B,k,G}}a_{j_1,j_2}\frac {\phi_0\left(\frac D{e^r}\right)B_N^{
(q)}\left(\sum_{K|t_1^2-4,D,\gcd\left(K,kG\right)=1}\mu\left(K\right
)\right)}{D^{1+s}}\Sigma ,\label{915}\end{equation}
where $\Sigma$ abbreviates
\[e\left(\frac {NC\overline {kD}}{\gamma}\right)\sum_{\lambda |D,
N}\frac {\overline {\epsilon_{D/\lambda}}}{\sqrt {\frac D{\lambda}}}\left
(\frac {t_1^2-4}{\lambda}\right)T\left(\left(t_1^2-4\right)\left(\frac
N{\lambda}\right)^2,\overline {4\gamma^2k^2}\left(t_2^2-4\right);\frac
D{\lambda}\right).\]
Indeed, we used here that because of the presence of the
factor $\left(\frac {t_1^2-4}D\right)$ on the left-hand side of \eqref{911} we
must have $\gcd\left(t_1^2-4,D\right)=1$, hence we can insert the
factor $\sum_{K|t_1^2-4,D,\gcd\left(K,kG\right)=1}\mu\left(K\right
)$, since we have $\gcd\left(D,kG\right)=1$.

We can write \eqref{915} as
\begin{equation}\sum_{j_1,j_2,\,t_1,t_2,D}a_{j_1,j_2}\frac {\phi_
0\left(\frac D{e^r}\right)}{D^{1+s}}\left(\sum_{K|t_1^2-4,D,\gcd\left
(K,kG\right)=1}\mu\left(K\right)\right)\sum_{\lambda |D}\frac {\overline {
\epsilon_{D/\lambda}}}{\sqrt {\frac D{\lambda}}}\left(\frac {t_1^
2-4}{\lambda}\right)S_{\lambda},\label{916}\end{equation}
where $S_{\lambda}$ abbreviates
\begin{equation}\sum_{N\in\mathbb{Z},N\equiv 0\left({\rm m}{\rm o}
{\rm d}\,\lambda\right)}B_N^{(q)}e\left(\frac {NC\overline {kD}}{
\gamma}\right)T\left(\left(t_1^2-4\right)\left(\frac N{\lambda}\right
)^2,\overline {4\gamma^2k^2}\left(t_2^2-4\right);\frac D{\lambda}\right
),\label{916,05)}\end{equation}
and we sum in \eqref{916} under the conditions \eqref{912}, \eqref{913}.

We want to show that the contribution of $K\gg X^{\xi}$ to
\eqref{916} is admissible for any $\xi >0$. In order to do that we now
transform back $S_{\lambda}$ by the Poisson formula. (We will
apply it only for the part $K\gg X^{\xi}$, therefore we do not
get the same expression after the two applications of
the Poisson formula.) The arithmetic weights depend on
the residue of $N$ modulo $D\gamma$, therefore $S_{\lambda}$ equals
\begin{equation}\frac 1{\gamma D}\sum_{u\in\mathbb{Z}}C_uA_u,\label{916,1}\end{equation}
where $C_u$ denotes
\[\int_{-\infty}^{\infty}\left(\int_{-\infty}^{\infty}\left(\left
(t_1^2-4\right)\left(t_2^2-4\right)\right)^sG_q\left(h,t_2\right)
e\left(-\frac t{\gamma D}h\right)dh\right)e\left(-\frac u{\gamma
D}t\right)dt\]
and $A_u$ denotes
\[\sum_{N\:{\rm m}{\rm o}{\rm d}\:\gamma D,\lambda |N}e\left(\frac {
Nu}{\gamma D}\right)e\left(\frac {NC\overline {kD}}{\gamma}\right
)T\left(\left(t_1^2-4\right)\left(\frac N{\lambda}\right)^2,\overline {
4\gamma^2k^2}\left(t_2^2-4\right);\frac D{\lambda}\right).\]
We have by Fourier inversion that
\begin{equation}C_u=\gamma D\left(\left(t_1^2-4\right)\left(t_2^2
-4\right)\right)^sG_q\left(-u,t_2\right).\label{916,2}\end{equation}
By the definition of the Salié sums we have (writing
$N_1=N/\lambda$, $D_1=D/\lambda$) that $A_u$ equals
\begin{equation}\sum_{x\:{\rm m}{\rm o}{\rm d}\:D_1,\gcd\left(x,D_
1\right)=1}\left(\frac x{D_1}\right)e\left(\frac {\overline {4\gamma^
2k^2}\left(t_2^2-4\right)x}{D_1}\right)R_x\label{917}\end{equation}
with
\[R_x:=\sum_{N_1\:{\rm m}{\rm o}{\rm d}\:\gamma D_1}e\left(\frac {
N_1u}{\gamma D_1}\right)e\left(\frac {N_1C\overline {kD_1}}{\gamma}\right
)e\left(\frac {\overline x\left(t_1^2-4\right)N_1^2}{D_1}\right).\]
Since we have $\gcd\left(\gamma ,D_1\right)=1$, we can write $N_1
=a\gamma +bD_1$,
where $a$ runs modulo $D_1$ and $b$ runs modulo $\gamma$. Then
\[R_x=\sum_{b\:{\rm m}{\rm o}{\rm d}\:\gamma}e\left(\frac {bu}{\gamma}\right
)e\left(\frac {bC\overline k}{\gamma}\right)\sum_{a\:{\rm m}{\rm o}
{\rm d}\:D_1}e\left(\frac {\overline x\left(t_1^2-4\right)a^2\gamma^
2}{D_1}\right)e\left(\frac {au}{D_1}\right).\]
Then we see that $R_x=0$ if $ku+C$ is not divisible by $\gamma$. Let
\[\mbox{\rm $z:=\gcd$$\left(t_1^2-4,D_1\right)$},\]
then we also see that $R_x=0$ if $u$ is not divisible by $z$. If $
\gamma$$|ku+C$ and $z|u$, then
\[R_x=\gamma\sum_{a\:{\rm m}{\rm o}{\rm d}\:D_1}e\left(\frac {\overline
x\frac {t_1^2-4}za^2\gamma^2}{D_1/z}\right)e\left(\frac {au/z}{D_
1/z}\right),\]
and we get by \eqref{400} that
\[R_x=\gamma z\epsilon_{D_1/z}\sqrt {D_1/z}\left(\frac {x\frac {t_
1^2-4}z}{D_1/z}\right)e\left(-\frac {\frac {u^2x}{z^2}\overline {
4\frac {t_1^2-4}z\gamma^2}}{D_1/z}\right).\]
Hence \eqref{917} gives that $A_u=0$ unless $\gamma$$|ku+C$ and $
z|u$
hold, and if these two conditions are true, then $A_u$ equals
the product of
\[\gamma\epsilon_{D_1/z}\sqrt {D_1z}\left(\frac {\frac {t_1^2-4}z}{
D_1/z}\right)\]
and
\begin{equation}\sum_{x\:{\rm m}{\rm o}{\rm d}\:D_1,\gcd\left(x,D_
1\right)=1}\left(\frac xz\right)e\left(\frac {\overline {4\gamma^
2k^2}\left(t_2^2-4\right)x}{D_1}\right)e\left(-\frac {\frac {u^2x}{
z^2}\overline {4\frac {t_1^2-4}z\gamma^2}}{D_1/z}\right).\label{918}\end{equation}
Let $D_2=\gcd\left(D_1,z^{\infty}\right)$ and $D_1=D_2D_3$, then we clearly have
$\gcd\left(D_2,D_3\right)=1$, $z|D_2$ and every prime divisor of $
D_2$ divides $z$.
Let us write $x=AD_2+BD_3$, then \eqref{918} equals the product of
\begin{equation}\sum_{A\:{\rm m}{\rm o}{\rm d}\:D_3,\gcd\left(D_3
,A\right)=1}e\left(\frac {\overline {4\gamma^2k^2}\left(t_2^2-4\right
)A}{D_3}\right)e\left(-\frac {\frac {u^2}z\overline {4\frac {t_1^
2-4}z\gamma^2}A}{D_3}\right)\label{919}\end{equation}
and
\begin{equation}\sum_{B\:{\rm m}{\rm o}{\rm d}\:D_2,\gcd\left(D_2
,B\right)=1}\left(\frac {BD_3}z\right)e\left(\frac {\overline {4\gamma^
2k^2}\left(t_2^2-4\right)B}{D_2}\right)e\left(-\frac {\frac {u^2}{
z^2}\overline {4\frac {t_1^2-4}z\gamma^2}B}{\frac {D_2}z}\right).\label{920}\end{equation}
It is clear that \eqref{919} equals
\begin{equation}\sum_{A\:{\rm m}{\rm o}{\rm d}\:D_3,\gcd\left(D_3
,A\right)=1}e\left(\frac {\left(\left(t_2^2-4\right)\left(t_1^2-4\right
)-k^2u^2\right)A}{D_3}\right).\label{921}\end{equation}
We write $B=B_0+jz$ in \eqref{920}, here $B_0$ runs modulo
$z$ such that
\[\gcd\left(B_0,z\right)=1,\]
and $j$ runs modulo $D_2/z$. Then \eqref{920} equals
\[\sum_{B_0\:{\rm m}{\rm o}{\rm d}\:z,\gcd\left(z,B_0\right)=1}\left
(\frac {B_0D_3}z\right)U_{B_0},\]
where $U_{B_0}$ denotes
\begin{equation}\sum_{j\:{\rm m}{\rm o}{\rm d}\:D_2/z}e\left(\frac {\overline {
4\gamma^2k^2}\left(t_2^2-4\right)\left(B_0+jz\right)}{D_2}\right)
e\left(-\frac {\frac {u^2}{z^2}\overline {4\frac {t_1^2-4}z\gamma^
2}\left(B_0+jz\right)}{\frac {D_2}z}\right).\label{930}\end{equation}
One can see that \eqref{930} equals 0 if $D_2/z$ does not
divide
\[\frac {t_1^2-4}z\left(t_2^2-4\right)-\frac {k^2u^2}z,\]
i.e. if $D_2$ does not divide $\left(t_1^2-4\right)\left(t_2^2-4\right
)-k^2u^2$. In the case
\[D_2|\left(t_1^2-4\right)\left(t_2^2-4\right)-k^2u^2\]
we use the trivial estimate that \eqref{920} is $\ll D_2$.

Let us write $D_4:=\gcd\left(\left(t_1^2-4\right)\left(t_2^2-4\right
)-k^2u^2,D_3\right)$. Then it
is well-known that the Ramanujan sum \eqref{921} equals
\begin{equation}\mu\left(\frac {D_3}{D_4}\right)\frac {\phi\left(
D_3\right)}{\phi\left(D_3/D_4\right)},\label{959}\end{equation}
hence \eqref{921} is $\ll D_4$.

By the above reasoning we get that $A_u=0$ unless $z|u$ and
\[D_2|\left(t_1^2-4\right)\left(t_2^2-4\right)-k^2u^2\]
hold, and if these two conditions are true, then
\[A_u\ll\gamma\sqrt {D_1z}D_2D_4.\]
Hence using also \eqref{916}, \eqref{916,1} and \eqref{916,2}
we get that the contribution of the part of \eqref{916}
belonging to a given $K\ge 1$ is $\ll$ than
\begin{equation}\gamma\sum_{z\ge 1,K|z}\sum_{\,t_1,t_2,D,\;z|t_1^
2-4,D}\frac {\phi_0\left(\frac D{e^r}\right)}D\sum_{D_2,D_3|D,\gcd\left
(D_2,D_3\right)=1}\sum_{D_4|D_3}T\label{1000}\end{equation}
with
\[T:=\sum_{u\in\mathbb{Z},D_2D_4|\left(t_1^2-4\right)\left(t_2^2-
4\right)-k^2u^2,z|u}\sqrt {z}D_2D_4\left|G_q\left(-u,t_2\right)\right
|.\]
We used here that since we have the factor $\left(\frac {t_1^2-4}{
\lambda}\right)$ in
\eqref{916}, we may assume $\gcd\left(t_1^2-4,\lambda\right)=1$, hence
$z=\gcd\left(t_1^2-4,D\right)$. And we have $K|t_1^2-4,D$, so $K|
z$.

We now use \eqref{913,1}, \eqref{913,2}, \eqref{602,1},
\eqref{602,2}, \eqref{79,1} and \eqref{79,2}. We use also
that for given $t_1,t_2$ and $u$ the number of possible $D_2,D_4$
is $\ll X^{\epsilon}$. Fixing $D_2$ and $D_4$ we must have $D=D_2
D_4j$
with some integer $j$, and for fixed $j$ the number of
possible $D_3$ is $\ll X^{\epsilon}$. We get in this way that if $
\epsilon >0$ is
given and $\beta_1$, $\alpha$ and $\delta$ are small enough in terms of $
\epsilon$,
then the part of \eqref{916} belonging to $K>X^{\xi}$ is $\ll$ than
\begin{equation}X^{\epsilon}\left(\frac dX\right)^{3/2}\sum_{K>X^{
\xi}}\sum_{z\ge 1,K|z}\sum_{\,t_1,t_2,\;z|t_1^2-4}\sum_{u\in\mathbb{
Z},z|u}\sqrt {z}E_X\left(t_1,t_2,ku\right),\label{1001}\end{equation}
where $E_X\left(t_1,t_2,ku\right)$ denotes $B_{X,i,j_1,j_2}\left(
t_1,t_2,ku\right)$ in the case $q=1$, and
$E_X\left(t_1,t_2,ku\right)$ denotes $D_{X,i,j_1,j_2}\left(t_1,t_
2,ku\right)$ in the case $q=2$.

On the one hand, we have for every $1\le z\ll d$ that
\begin{equation}\sum_{\,t_1,t_2,\;z|t_1^2-4}\sum_{u\in\mathbb{Z},
z|u}\sqrt {z}E_X\left(t_1,t_2,ku\right)\ll X^{\epsilon}\frac {dX}{
z^{1+c_0}},\label{1002}\end{equation}
on the other hand, we have
\begin{equation}\sum_{K>X^{\xi}}\sum_{z\ge d,K|z}\sum_{\,t_1,t_2,\;
z|t_1^2-4}\sum_{u\in\mathbb{Z},z|u}\sqrt {z}E_X\left(t_1,t_2,ku\right
)\ll X^{\epsilon +\frac 32}\ll\frac {dX}{X^{c_0}}\label{985}\end{equation}
with some absolute constant $c_0>0$ if $\beta_1$, $\alpha$ and $\delta$ are
small enough in terms of $\epsilon$. We can see \eqref{1002} considering
separately the cases $z\ll\sqrt {X}$, $\sqrt {X}\ll z\ll d$. We now show
\eqref{985}. For given $z\ge d$ the number of possible $K,$ $t_1$ and $
u$ is $\ll X^{\epsilon}$,
and the number of possible $t_2$ is $\ll X^{\epsilon +\frac 12}$. The
number of such $z\ge d$ for which there is at least one $t_1$
such that $z|t_1^2-4$ is $\ll X^{\epsilon +\frac 12}$, and for every such $
z$ we
have $\sqrt {z}\ll X^{\epsilon +\frac 12}$. Indeed, this follows from the fact that
all the possible $t_1$ are $\ll X^{\epsilon +\frac 12}$. These considerations
prove \eqref{985}.

We get from \eqref{1001}, \eqref{1002} and \eqref{985} that the contribution of the part of \eqref{916}
belonging to $K>X^{\xi}$ is admissible for any $\xi >0$.

It remains to prove that the part $K\le X^{\xi}$, $N=0$ of
\eqref{916} is also admissible. We first note that in
general we have that
\[\left|T\left(0,n;c\right)\right|^2=\sum_{x,y\:{\rm m}{\rm o}{\rm d}\:
c}\left(\frac xc\right)\overline {\left(\frac yc\right)}e\left(\frac {
nx-ny}c\right).\]
We can write here $x\equiv ky$ with some $\gcd\left(k,c\right)=1$, and
then $\left|T\left(0,n;c\right)\right|^2$ equals
\[\sum_{k\:{\rm m}{\rm o}{\rm d}\:c,\gcd\left(k,c\right)=1}\left(\frac
kc\right)\sum_{y\:{\rm m}{\rm o}{\rm d}\:c,\gcd\left(y,c\right)=1}
e\left(\frac {nky-ny}c\right).\]
As below \eqref{959} we estimate the
Ramanujan sum by $\gcd\left(n\left(k-1\right),c\right)$, and we get
\[\left|T\left(0,n;c\right)\right|^2\ll\gcd\left(n,c\right)\sum_{
k\:{\rm m}{\rm o}{\rm d}\:c}\gcd\left(k-1,c\right)\ll_{\epsilon}c^{
1+\epsilon}\gcd\left(n,c\right)\]
for any $\epsilon >0$. Therefore the part $K\le X^{\xi}$, $N=0$ of
\eqref{916} can be estimated by
\begin{equation}X^{\epsilon}\sum_{\,t_1,t_2,D}\frac {\phi_0\left(\frac
D{e^r}\right)}D\sqrt {\gcd\left(t_2^2-4,D\right)}\sum_{j_1,j_2}a_{
j_1,j_2}B_0^{(q)}\label{724}\end{equation}
for any $\epsilon >0$. We used here that only $B_0^{(q)}$ depends on $
j_1$
and $j_2$. Using Lemma \ref{lem:11,95} we get that if $\alpha$ and $
\delta$ are
small enough in terms of $\epsilon$, then \eqref{724} can be estimated by
\begin{equation}X^{\epsilon}\frac {d^{5/2}}{X^{3/2}}\sum_{\,t_1,t_
2,D}\frac {\phi_0\left(\frac D{e^r}\right)A_X\left(t_1,t_2\right)
\prod_{k=1}^2f_{\beta_1}\left(\frac {t_k-a_i}{a_{i+1}-a_i}\right)}
D\sqrt {\gcd\left(t_2^2-4,D\right)}G_1\label{843}\end{equation}
for $q=1$, and by

\begin{equation}X^{\epsilon}\frac {d^{5/2}}{X^{3/2}}\sum_{\,t_1,t_
2,D}\frac {\phi_0\left(\frac D{e^r}\right)C_X\left(t_1,t_2\right)
\prod_{k=1}^2f_{\beta_1}\left(\frac {t_k-a_i}{a_{i+1}-a_i}\right)}
D\sqrt {\gcd\left(t_2^2-4,D\right)}G_2\label{844}\end{equation}
for $q=2$, with
\[G_1:=\sum_{j_1,j_2}a_{j_1,j_2}\left(B^2\left(S_0,T_0\right)-1\right
)^sh_{\beta_1,C_1\left(j_1,j_2\right)}\left(0\right),\]
\[G_2:=\sum_{j_1,j_2}a_{j_1,j_2}\left(A^2\left(S_0,T_0\right)-1\right
)^sh_{\beta_1,C_2\left(j_1,j_2\right)}\left(0\right)\]
for any $\epsilon >0$. Using trivial estimates for $G_1,G_2$ we could obtain easily
that \eqref{843} and \eqref{844} are $\ll X^{\epsilon}\frac {d^{5
/2}}{X^{1/2}}$ for any
$\epsilon >0$ if $\alpha$ and $\delta$ are small enough in terms of $
\epsilon$. Indeed,
for given $t_2$ there are $\ll X^{\epsilon}$ possible values of
$\gcd\left(t_2^2-4,D\right)$, and we have $D=j\gcd\left(t_2^2-4,D\right
)$ with an
integer $j$.

Hence it is enough to give estimates $G_1,G_2\ll X^{-\beta_0}$ with
some absolute constant $\beta_0>0$ for large enough $J$. In the case $
q=2$ it is important to use for this estimate that we
have $\left|t_1-t_2\right|\gg X^{\frac 12-\alpha}$, which follows from \eqref{678}.

But these estimates follow at once from \eqref{888} and the easily seen facts that
\begin{equation}\left(B^2\left(S_0\left(j_1,t_1\right),T_0\left(j_
2,t_2\right)\right)-1\right)^s-\left(B^2\left(S_0\left(0,t_1\right
),T_0\left(0,t_2\right)\right)-1\right)^s,\label{6711}\end{equation}
\begin{equation}\left(A^2\left(S_0\left(j_1,t_1\right),T_0\left(j_
2,t_2\right)\right)-1\right)^s-\left(A^2\left(S_0\left(0,t_1\right
),T_0\left(0,t_2\right)\right)-1\right)^s,\label{6712}\end{equation}
\begin{equation}h_{\beta_1,C_1\left(j_1,j_2\right)}\left(0\right)
-h_{\beta_1,C_1\left(0,j_2\right)}\left(0\right),\label{6713}\end{equation}
\begin{equation}h_{\beta_1,C_2\left(j_1,j_2\right)}\left(0\right)
-h_{\beta_1,C_2\left(0,j_2\right)}\left(0\right)\label{6714}\end{equation}
are $\ll X^{-\beta_0}$ with some absolute constant $\beta_0>0$ for any $
j_1$
and $j_2$, if $\alpha$, $\beta_1,\beta_2$ and $\delta$ are small enough. The
estimates \eqref{6711} and \eqref{6712} are easy to see
using \cite[(6.28), (6.29)]{Biro} (in the case of \eqref{6712} we use \eqref{741}). For the proof
of \eqref{6713} and \eqref{6714} we use that
\[C_1\left(j_1,j_2\right)_{}-\frac {\left(T_0\left(0,t_2\right)-T_
0\left(j_2,t_2\right)\right)\frac d{dT}B\left(S_0\left(0,t_1\right
),T\right)_{|T=T_0\left(0,t_2\right)}-\frac d{a_i^2}X^{\delta}}{d\left
(\left(t_1^2-4\right)\left(t_2^2-4\right)\right)^{-1/2}}\]
and
\[C_2\left(j_1,j_2\right)_{}-\frac {\left(T_0\left(0,t_2\right)-T_
0\left(j_2,t_2\right)\right)\frac d{dT}A\left(S_0\left(0,t_1\right
),T\right)_{|T=T_0\left(0,t_2\right)}-\frac {d\left|t_2-t_1\right
|X^{\delta}}{a_i\left(X-a_i^2\right)}}{d\left(\left(t_1^2-4\right
)\left(t_2^2-4\right)\right)^{-1/2}}\]
are $\ll X^{-\beta_0}$. To see these estimates we use our
conditions, \cite[(6.26)]{Biro} and the formulas given for the derivatives of
$B\left(S,T\right)$ and $A\left(S,T\right)$ given in the lines below \cite[(6.27)]{Biro}. The lemma is proved.
\end{proof}

\subsection{The final result}

It is easy to see from the following two lemmas that
Conjecture \ref{conj:1} implies \eqref{2.2}. Hence Theorem
\ref{thm:2} also follows from the next two lemmas.

\begin{lemma}\label{lem:13}Assume \eqref{201}. Then we
have that \eqref{2.2} is true with a positive $\beta$ uniformly for $
1\le\tau\le 2$ if $J$ is large enough, assuming that the following statement
is true:

{\bf Statement.} There are numbers $\delta ,\epsilon >0$ such that for every
positive integers $K,L\ll X^{\epsilon}$ satisfying $\gcd\left(K,L\right
)=1$ and
that $L=4u^2$ with some positive integer $u$ divisible by $4$, for every integers $
r_1,r_2,r_3,r_4$ such that
$\gcd\left(r_4,u\right)=1$, for every real numbers $C,a,M$ such that $
1\le C\ll X$, $1\ll M\ll CX^{\epsilon}/d$ and
\begin{equation}a,\sqrt {X}-2a\gg X^{\frac 12-\epsilon},\label{3714,4}\end{equation}
for every $\kappa\in \{-1,1\}$, for every smooth function $g\left
(t_1,t_2,m,c\right)$ which is supported in
the set
\begin{equation}a\le t_1,t_2\le 2a,\:M\le m\le 2M,\:C\le c\le 2C\label{3714,5}\end{equation}
and satisfies
\[a^{i_1+i_2}M^{i_3}C^{i_4}\frac {\partial^{i_1+i_2+i_3+i_4}g}{\partial
t_1^{i_1}\,\partial t_2^{i_2}\,\partial m^{i_3}\,\partial c^{i_4}}
\ll_{i_1,i_2,i_3,i_4}X^{\epsilon\left(i_1+i_2+i_3+i_4\right)}\]
for every nonnegative integers $i_1,i_2,i_3,i_4$, writing
\[f\left(t_1,t_2,m,c\right):=g\left(t_1,t_2,m,c\right)e\left(\frac {
m\sqrt {t_1^2-4}\sqrt {t_2^2-4}}{cu}h_{\kappa}\left(t_1,t_2\right
)\right)\]
with
\[h_{\kappa}\left(t_1,t_2\right):=\frac {X+\kappa\sqrt {X-\left(t_
1^2-4\right)}\sqrt {X-\left(t_2^2-4\right)}}{\sqrt {t_1^2-4}\sqrt {
t_2^2-4}},\]
we have that
\begin{equation}\sum_{t_1,t_2,m,c}f\left(t_1,t_2,m,c\right)T\left
(\left(t_1^2-4\right)m^2,\overline L\left(t_2^2-4\right);c\right)
\ll_{\epsilon ,\delta}C^{3/2}X^{1-\delta},\label{6715}\end{equation}
where we sum under the conditions $K|t_1^2-4,c$ and
\[\,t_1\equiv r_1\left({\rm m}{\rm o}{\rm d}\,u\right),\,t_2\equiv
r_2\left({\rm m}{\rm o}{\rm d}\,u\right),\,\,m\equiv r_3\left({\rm m}
{\rm o}{\rm d}\,u\right),\,c\equiv r_4\left({\rm m}{\rm o}{\rm d}\,
u\right).\]
\end{lemma}
\begin{proof}We apply Lemmas \ref{lem:11,941} and
\ref{lem:11,94}. In \eqref{971} we write $L:=4\gamma^2k^2$, $c:=\frac
D{\lambda}$,
$m:=\frac N{\lambda}$, and we apply the Statement given in the lemma
fixing $K$, $\lambda$, $j_1$ and $j_2$. We express $B_N^{(q)}$ by Lemma \ref{lem:11,95}. We write the factor
\[e\left(-\sigma\frac N{kD\gamma}\sqrt {t_1^2-4}\sqrt {t_2^2-4}A\left
(S_0\left(j_1,t_1\right),T_0\left(j_2,t_2\right)\right)\right)\]
as the product of
\[e\left(-\sigma\frac N{kD\gamma}\sqrt {t_1^2-4}\sqrt {t_2^2-4}h_{
-1}\left(t_1,t_2\right)\right)\]
and
\begin{equation}e\left(-\sigma\frac N{kD\gamma}\sqrt {t_1^2-4}\sqrt {
t_2^2-4}\left(A\left(S_0\left(j_1,t_1\right),T_0\left(j_2,t_2\right
)\right)-h_{-1}\left(t_1,t_2\right)\right)\right),\label{1007}\end{equation}
and we consider \eqref{1007} as a factor of the smooth
function $g\left(t_1,t_2,m,c\right)$. We replace similarly
$B\left(S_0\left(j_1,t_1\right),T_0\left(j_2,t_2\right)\right)$ by $
h_1\left(t_1,t_2\right)$. The lemma follows
by carefully examining the definitions and estimating
trivially.

\end{proof}

\begin{lemma}\label{lem:14}
Conjecture \ref{conj:1} implies the Statement in Lemma \ref{lem:13}.
\end{lemma}

\begin{proof}In the case when $\left(t_1^2-4\right)\left(t_2^2-4\right
)$ is not a
square we estimate the summation over $c$ in \eqref{6715}
by Conjecture \ref{conj:1}.

Now assume that $\left(t_1^2-4\right)\left(t_2^2-4\right)$ is a square with some
integers $a\le t_1,t_2\le 2a$, where $a$ satisfies \eqref{3714,4}.
Let us write
\[\left(t_1^2-4\right)\left(t_2^2-4\right)=y^2\]
with some integer $y>0$. The left-hand side here equals
$t_1^2t_2^2+O\left(a^2\right)$. So
\[y^2-t_1^2t_2^2=\left(y-t_1t_2\right)\left(y+t_1t_2\right)=O\left
(a^2\right).\]
Using \eqref{3714,4} we then get that $y-t_1t_2\ll 1$.

If $y=t_1t_2+A$, then we get
\[16-4t_1^2-4t_2^2=2t_1t_2A+A^2.\]
This shows that for given $t_2$ and $A$ we can have at most
two numbers $t_1$. Therefore, for given $a\le t_2\le 2a$ the
number of integers $t_1$ satisfying that $a\le t_1\le 2a$ and $\left
(t_1^2-4\right)\left(t_2^2-4\right)$ is a
square is $\ll 1$.

Using \cite[Lemma 4.18]{J} we have that
\[T\left(\left(t_1^2-4\right)m^2,\overline L\left(t_2^2-4\right);
c\right)\ll\left(\gcd\left(t_2^2-4,c\right)\right)^{1/2}c^{\frac
12+\epsilon}.\]
Therefore the contribution of those pairs
$t_1$, $t_2$ for which $\left(t_1^2-4\right)\left(t_2^2-4\right)$ is a square is
\begin{equation}\ll X^{\epsilon}MC^{\frac 12}\sum_{a\le t_2\le 2a
,\,C\le c\le 2C}\left(\gcd\left(t_2^2-4,c\right)\right)^{1/2}\label{6716}\end{equation}
in \eqref{6715}. It is clear that
\[\left(\gcd\left(t_2^2-4,c\right)\right)^{1/2}\ll\gcd\left(t_2-2
,c\right)+\gcd\left(t_2+2,c\right),\]
and we finally get that \eqref{6716} is
\[\ll X^{2\epsilon}MC^{\frac 32}a.\]
Since $M\ll X^{1+\epsilon}/d$ and $a\ll\sqrt X$, by \eqref{201} we see that
this contribution is acceptable in \eqref{6715}. The lemma
is proved.

\end{proof}


\begin{thebibliography}{BDFHHL}

\bibitem[A]{A} https://github.com/KokunoYumeto/s6-complex-structure-workbench/

  blob/main/pro\_returns/S6\_05\_ARITHMETIC\_final\_20261115T1846Z/

  S6\_05\_ARITHMETIC/S6\_05\_ARITHMETIC\_paper.pdf

\bibitem[B1]{Biro} \textsc{A. Biró}, Local square mean in the hyperbolic circle problem,

\emph{arXiv:2403.16113v2}, to appear in Algebra and Number Theory

\bibitem[B2]{biroClass} \textsc{A. Biró} On the class number of pairs of binary quadratic forms,
\emph{J. Théor. Nombres Bordeaux} \textbf{37} (2025), no.~3, 897--924.

\bibitem[BFR]{BFR} \textsc{O. Balkanova, D. Frolenkov, M.S. Risager} Prime geodesics and averages of the Zagier $L$-series,
\emph{ Math. Proc. Cambridge Philos. Soc.} 172 (2022), no. 3, 705--728.



\bibitem[B-K-S]{Br} \textsc{T.D. Browning, V.V. Kumaraswamy, R.S. Steiner} Twisted Linnik implies Optimal Covering Exponent for $S^3$,
\emph{International Mathematics Research Notices} Vol. 2019, no. 1, 140--164.

\bibitem[CCLR]{CCLR} \textsc{D. Chatzakos, G. Cherubini, S. Lester, M.S. Risager}, The hyperbolic circle problem over Heegner points,
\emph{arXiv:2506.13883v1}


\bibitem[G]{Granville} \textsc{A. Granville}, Smooth numbers: computational number theory and beyond, Algorithmic number theory: lattices, number fields, curves and cryptography, Math. Sci. Res. Inst. Publ., vol. 44, Cambridge Univ. Press, Cambridge, 2008, pp. 267–323



\bibitem[I]{I} \textsc{H. Iwaniec}, Topics in
classical automorphic forms, American Mathematical Society, 1997.


\bibitem[J]{J} E. Jones, \textit{The Kloosterman circle method and weighted representation numbers of quadratic forms}, Res. Math. Sci. \textbf{12} (2025), no. 59.


\bibitem[L]{L} \textsc{B. Louvel} The first moment of Sali\'e sums, \emph{Monatshefte f\"ur Mathematik} Vol. 168, 523--543 (2012).


\bibitem[PR]{PR} \textsc{Y.N. Petridis, M.S. Risager} Local average in hyperbolic lattice point counting, with an appendix by Niko Laaksonen,
\emph{ Math. Z.} 285 (2017), no. 3-4, 1319--1344.

\bibitem[Se]{Serre} \textsc{J.-P. Serre}, A Course in Arithmetic, Graduate Texts in Mathematics, vol. 7, Springer, 1973

\bibitem[St]{St} \textsc{R.S. Steiner} On a twisted version of Linnik and Selberg conjecture on sums of Kloosterman sums,
\emph{Mathematika} \textbf{65} (2019), no.~3, 437--474.




\end{thebibliography}
\end{document}